\documentclass[a4paper,10pt,leqno]{amsart}
        \title[On the group cohomology of the semi-direct product  $\IZ^n \rtimes_{\rho} \IZ/m$]
        {On the group cohomology of the semi-direct product $\IZ^n \rtimes_{\rho} \IZ/m$ and a conjecture of Adem-Ge-Pan-Petrosyan}
       \author{Martin Langer}
       \author{Wolfgang L\"uck}
        \address{Rheinische Wilhelms-Universit\"at Bonn\\
               Mathematisches Institut\\
               Endenicher Allee 62, 53115 Bonn, Germany}
          \email{martinlanger@yahoo.com}
          \email{wolfgang.lueck@him.uni-bonn.de}
      \urladdr{http://www.math.uni-bonn.de/people/mlanger}
      \urladdr{http://www.him.uni-bonn.de/lueck}
         \date{May, 2011}
     \keywords{Tate and group (co-)homology, split extensions of finitely generated free abelian groups by finite cyclic groups.}
    \subjclass[2010]{20J06, 55T10}

\usepackage{pdfsync}
\usepackage{calc}
\usepackage{enumerate,amssymb}
\usepackage{color}
\usepackage[arrow,curve,matrix,tips,2cell]{xy}
  \SelectTips{eu}{10} \UseTips
  \UseAllTwocells
\usepackage{tikz}
\usepackage{pdfcolmk}

\DeclareMathAlphabet{\matheurm}{U}{eur}{m}{n}
\DeclareMathOperator{\aut}{aut}

\DeclareMathOperator{\ch}{ch}

\DeclareMathOperator{\Ext}{Ext}

\DeclareMathOperator{\id}{id}

\DeclareMathOperator{\pt}{pt}
\DeclareMathOperator{\pr}{pr}

\DeclareMathOperator{\Sw}{Sw}

  \newcommand{\IQ}{\mathbb{Q}}
  \newcommand{\IR}{\mathbb{R}}

  \newcommand{\IZ}{\mathbb{Z}}

  \newcommand{\cala}{\mathcal{A}}

  \newcommand{\calm}{\mathcal{M}}
  
  \newcommand{\calo}{\mathcal{O}}
  \newcommand{\calp}{\mathcal{P}}

\newcounter{commentcounter}

 \newcommand{\dual}{^{\wedge}}

\theoremstyle{plain}
\newtheorem{theorem}{Theorem}[section]

\newtheorem{lemma}[theorem]{Lemma}

\newtheorem{corollary}[theorem]{Corollary}
\newtheorem{proposition}[theorem]{Proposition}
\newtheorem{conjecture}[theorem]{Conjecture}

\newtheorem*{theorem*}{Theorem}

\theoremstyle{definition}
\newtheorem{definition}[theorem]{Definition}

\newtheorem{example}[theorem]{Example}

\newtheorem{remark}[theorem]{Remark}

\theoremstyle{remark}

\makeatletter\let\c@equation=\c@theorem\makeatother

\newcommand{\ev}{\mathit{ev}}
\newcommand{\odd}{\mathit{odd}}

\newcommand{\version}[1]                       %marks the date of last editing and compilation
{\begin{center} last edited on #1\\
last compiled on \today \\
name of texfile: \jobname
\end{center}
}

\newcommand{\xycomsquare}[8]                      % kommutatives Quadrat (xy-Version
{\xymatrix{#1 \ar[r]^{#2} \ar[d]^{#4} &
#3 \ar[d]^{#5}  \\
#6\ar[r]^{#7} &
#8
}
}

\newcommand{\xycomsquareminus}[8]                      % kommutatives Quadrat (xy-Version)
{\xymatrix{#1 \ar[r]^-{#2} \ar[d]^-{#4} &
#3 \ar[d]^-{#5}  \\
#6\ar[r]^-{#7} &
#8
}
}

\begin{document}

\begin{abstract}
  Consider the semi-direct product $\IZ^n \rtimes_{\rho} \IZ/m$. A conjecture of
  Adem-Ge-Pan-Petrosyan predicts that the associated
  Lyndon-Hochschild-Serre spectral sequence collapses. We prove this
  conjecture provided that the $\IZ/m$-action on $\IZ^n$ is free outside the
  origin. We disprove the conjecture in general, namely, we give an example with
  $n = 6$ and $m = 4$, where the second differential does not vanish.
\end{abstract}

\maketitle

\newlength{\origlabelwidth}
\setlength\origlabelwidth\labelwidth

%%%%%%%%%%%%%%%%%%%%%%%%%%%%%%%%%%%%%%%%%%%%%%%%%%%%%%%%%%%%%%%%%%%%%%%%
%%%%%%%%%%%%%%%%%%%%%%%%%%% Introduction %%%%%%%%%%%%%%%%%%%%%%%%%%%%%%%%%%%
%%%%%%%%%%%%%%%%%%%%%%%%%%%%%%%%%%%%%%%%%%%%%%%%%%%%%%%%%%%%%%%%%%%%%%%%

\typeout{----------------------------  Introduction ------------------------------------}
\section*{Introduction}
\label{sec:introduction}

Throughout this paper let $G \cong \IZ/m$ be a finite cyclic group of order $m$
and let $L\cong \IZ^n$ be a finitely generated free abelian group of rank $n$.
Let $\rho \colon G \to \aut_{\IZ}(L)$ be a group homomorphism. It puts the
structure of a $\IZ G$-module on $L$. Let $\Gamma$ be the associated semi-direct
product $L \rtimes_{\rho} G$. We will make the assumption that the $G$-action of $G$ on $L$ is free
outside the origin unless stated explicitly differently.

Here is a brief summary of our results. We will show that the Tate cohomology
$\widehat{H}^i(G;\Lambda^j(L))$ vanishes for all $i,j$ for which $i + j$ is odd.
This will be the key ingredient for computations of the topological $K$-theory of
the group $C^*$-algebra of $\Gamma$ which will be carried out in a different
paper, generalizing previous calculations of Davis-L\"uck in the special case
where $m$ is a prime. We determine the group cohomology of $\Gamma$ in high dimensions
using classifying spaces for proper actions. A conjecture due to Adem-Ge-Pan-Petrosyan says that 
the Lyndon-Hochschild-Serre spectral sequence associated to the semi-direct product $L \rtimes_{\rho} G$ collapses.
We will prove it under the assumption mentioned above. Without this assumption we give counterexamples.

%%%%%%%%%%%%%%%%%%%%%%%%%%%%%%%%%%%%%%%%%%%%%%%%%%%%%%%%%%%%%%%%%%%%%%%%

\subsection{Tate cohomology}
\label{subsec:Tate_cohomology}

In the sequel $\Lambda^j =
\Lambda_{\IZ}^j$ stands for the $j$-th exterior power of a $\IZ$-module.

\begin{theorem}[Tate cohomology]
\label{the:Tate_cohomology}
Suppose that  the $G$-action on $L$ is free outside the origin, 
i.e., if for $g \in G$ and $x \in  L$ we have $gx = x$, then $g = 1$ or $x = 0$.

Then we get for the Tate cohomology 
$$\widehat{H}^i(G;\Lambda^j(L)) = 0$$
for all $i,j$ for which $i + j$ is odd.
\end{theorem}

Theorem~\ref{the:Tate_cohomology} will be proved in 
Section~\ref{sec:Proof_of_Theorem_the:Tate_cohomology}.

%%%%%%%%%%%%%%%%%%%%%%%%%%%%%%%%%%%%%%%%%%%%%%%%%%%%%%%%%%%%%%%%%%%%%%%%

\subsection{Motivation}
\label{subsec:Motivation}

Davis-L\"uck~\cite{Davis-Lueck(2010)} have computed the topological $K$-(co-)homology of $B\Gamma$ and $\underline{B}\Gamma$
and finally the topological $K$-theory of the reduced group $C^*$-algebra $C^*_r(\Gamma)$ in the special case
that $m$ is a prime. The result is:

\begin{theorem}[{\cite[Theorem~0.3]{Davis-Lueck(2010)}}]
\label{the:Davis-Lueck}
Suppose that $m = p$ for a prime $p$. Suppose that $G$ acts freely on $L$ outside the origin. 
Then

\begin{enumerate}
\item \label{the:Davis-Lueck:s}
There exists an integer $s$ uniquely determined by $(p-1)\cdot s = n$;

\item \label{the:Davis-Lueck:p_is_2} If $ p = 2$, then
$$K_m(C^*_r(\Gamma)) \cong
\begin{cases}
  \IZ^{3 \cdot 2^{n-1}} & m \; \text{even};
  \\
  0 & m \; \text{odd};
\end{cases}
$$
\item \label{the:Davis-Lueck_p_is_odd}
If $p$ is odd, then
$$K_m(C^*_r(\Gamma)) \cong
\begin{cases}
  \IZ^{d_{\ev}} & m \; \text{even};
  \\
  \IZ^{d_{\odd}} & m \; \text{odd},
\end{cases}
$$
where
\begin{eqnarray*}
  d_{\ev}
  & = &
  \frac{2^{(p-1)s} + p -1}{2p} + \frac{(p-1) \cdot p^{s-1}}{2} + (p-1) \cdot p^s;
  \\
  d_{\odd} 
  & = & 
  \frac{2^{(p-1)s} + p -1}{2p} -  \frac{(p-1) \cdot p^{s-1}}{2};
\end{eqnarray*}
\item \label{the:Davis-Lueck:torsionfree}
In particular $K_m(C^*_r(\Gamma))$ is always a finitely generated free abelian group.
\end{enumerate}
\end{theorem}

This computation is interesting in its own right but has also interesting
consequences.  For instance, the (unstable) Gromov-Lawson-Rosenberg Conjecture
holds for $\Gamma$ in dimensions $\ge 5$
(see~\cite[Theorem~0.7]{Davis-Lueck(2010)}). Davis-L\"uck are planning to apply
a version of Theorem~\ref{the:Davis-Lueck}  for the algebraic $K$- and $L$-theory of the integral group ring of $G$
to the classification of total spaces of certain torus bundles over lens
spaces. The starting point of the proof of Theorem~\ref{the:Davis-Lueck} is
Theorem~\ref{the:Tate_cohomology} in the special case $m = p$.

Recent work of Cuntz-Li~\cite{Cuntz-Li(2009integers)}
and~\cite{Cuntz-Li(2009function)} on the topological $K$-theory of
$C^*$-algebras arising from number theory triggered the question whether
Theorem~\ref{the:Davis-Lueck} can be extended to the general case, i.e., to the
case where $m$ is any natural number.  This question is also interesting in its
own right. But it can only be attacked if Theorem~\ref{the:Tate_cohomology}
holds, and this will be proved in this paper.  The situation relevant for the
work of Cuntz and Li is the case where $L$ is the ring of integers $\calo$ of
an algebraic number field $K$, $G$ is  the finite cyclic group $\mu$ of roots
of unity in $K^{\times}$,  and $\rho \colon \mu \to \aut(\calo)$ comes from the
multiplication in $\calo$. Obviously $\mu$ acts freely on $\calo$ outside the
origin.

%%%%%%%%%%%%%%%%%%%%%%%%%%%%%%%%%%%%%%%%%%%%%%%%%%%%%%%%%%%%%%%%%%%%%%%%

\subsection{Group cohomology}
\label{subsec:group_cohomology}

We will compute the group cohomology of $\Gamma$ in sufficiently large dimensions by using
the classifying space for proper actions, namely, we will prove in
Section~\ref{sec:The_cohomology_of_Gamma}:

\begin{theorem}\label{the:cohomology_of_Gamma}
Suppose that $G$ acts freely on $L$ outside the origin. 
Let $\calm$ be a complete system of representatives of the conjugacy
classes of maximal finite subgroups of $\Gamma$.
Then we obtain for $2k > n$ an isomorphism
\[ 
{H}^{2k}(\Gamma) \xrightarrow{\varphi^{2k}} \bigoplus_{(M) \in \calp} \widetilde{H}^{2k}(M),
\]
where $\varphi^{2k}$ is the map induced by the various inclusions $M \to \Gamma$ for $M \in \calm$.
For $2k+1 > n$  we get
\[ 
{H}^{2k+1}(\Gamma) \cong 0.
\]
\end{theorem}

%%%%%%%%%%%%%%%%%%%%%%%%%%%%%%%%%%%%%%%%%%%%%%%%%%%%%%%%%%%%%%%%%%%%%%%%

\subsection{On a conjecture of Adem-Ge-Pan-Petrosyan}
\label{subsec:On_a_Conjecture_of_Adem-Ge-Pan-Petrosyan}

We will analyze the following conjecture due to
Adem-Ge-Pan-Petrosyan~\cite[Conjecture~5.2]{Adem-Ge-Pan-Petrosyan(2008)}).

\begin{conjecture}[Adem-Ge-Pan-Petrosyan]
  \label{con:Adem-Ge-Pan-Petrosyan}
  The Lyndon-Hochschild-Serre spectral sequence associated to the semi-direct
  product $L \rtimes_{\rho} G$ collapses in the strongest sense, i.e., all
  differentials in the $E_r$-term for $r \ge 2$ are trivial and all extension
  problems at the $E_{\infty}$-level are trivial. In particular we get for all
  $k \ge 0$
  \[
  H^k(\Gamma;\IZ) \cong \bigoplus_{i + j = k} H^i(G;H^j(L)).
  \]
\end{conjecture}

This conjecture is known to be true if $m$ is squarefree 
(see~\cite[Corollary~4.2]{Adem-Ge-Pan-Petrosyan(2008)}) or if
there exists a so called \emph{compatible group action}
(see~Definition~\ref{def:compatible_group_action} and~\cite[Definition~2.1 and Theorem~2.3]{Adem-Pan(2006)}).

We will prove a positive and a negative result concerning 
Conjecture~\ref{con:Adem-Ge-Pan-Petrosyan}.

 \begin{theorem}[Free actions]\label{the:positive}
  Conjecture~\ref{con:Adem-Ge-Pan-Petrosyan} is true, provided that the 
 $G$-action on $L$ is free outside the origin. 
 \end{theorem}

The proof of Theorem~\ref{the:positive} will be given in Section~\ref{sec:Proof_of_Theorem_ref_the:positive}.

\begin{theorem}[Conjecture~\ref{con:Adem-Ge-Pan-Petrosyan} is not true in general]
\label{the:negative}
Consider the special case $n = 6$ and $m = 4$, where $\rho$ is given by the matrix
\[
\left(\begin{matrix} 
0 & 1 & 0 & 0 & 0 & 0
\\
-1 & 0 & 1 & 0 & 0 & 0 
\\
0 & 0 & 1 & 0 & 0 & 0
\\
0 & 0 & 0  & 1 & 0 & 0
\\
0 & 0 & 0  & -1 & 0 & 1
\\
0 & 0 & 0  & 0 & -1 & 0
\end{matrix}
\right)
\]
Then the second differential in the Lyndon-Hochschild-Serre spectral
sequence associated to the semi-direct product $L \rtimes_{\rho} G$ is
non-trivial. In particular Conjecture~\ref{con:Adem-Ge-Pan-Petrosyan}
is not true.
\end{theorem}

Theorem~\ref{the:negative} will be proved in Section~\ref{sec:A_counterexample}
based on the analysis of the cohomology classes $[\alpha_s]$ due to
Charlap-Vasquez~\cite{Charlap-Vasquez(1969)} presented in
Section~\ref{sec:The_cohomology_classes_[alpha_s]}.  These classes can be used to
describe the second differential in the Lyndon-Hochschild-Serre spectral
sequence and are obstructions to the existence of a compatible group action 
the sense of~\cite[Definition~2.1]{Adem-Pan(2006)} (see~Definition~\ref{def:compatible_group_action}).
The next result is an easy consequence
of Theorem~\ref{the:negative} and will be proved also in  Section~\ref{sec:A_counterexample}.

\begin{corollary} \label{cor:m_divisible_by_four}
\begin{enumerate}

\item  \label{cor:m_divisible_by_four:non-trivial}
If $m$ is divisible by four, we can find for $G \cong \IZ/m$ an $L$
such that  the second differential in the Lyndon-Hochschild-Serre spectral
sequence associated to the semi-direct product $L \rtimes_{\rho} G$ is
non-trivial;

\item   \label{cor:m_divisible_by_four:trivial}
If $m$ is not divisible by four, then for all  $L$
the  second differential in the Lyndon-Hochschild-Serre spectral
sequence associated to the semi-direct product $L \rtimes_{\rho} G$ is
trivial.

\end{enumerate}
\end{corollary}

 \begin{remark}[Reformulation of Conjecture~\ref{con:Adem-Ge-Pan-Petrosyan}] \label{rem:discussion_of_conjecture}
  In view of Corollary~\ref{cor:m_divisible_by_four} a very  optimistic guess is that Conjecture~\ref{con:Adem-Ge-Pan-Petrosyan}
  of Adem-Ge-Pan-Petrosyan  is  true if and only if  $m$ is not divisible by four. 
   \end{remark}

%%%%%%%%%%%%%%%%%%%%%%%%%%%%%%%%%%%%%%%%%%%%%%%%%%%%%%%%%%%%%%%%%%%%%%%%%

\subsection{Group cohomology and the equivariant Euler characteristic}
\label{subsec:cohomology_and_equivariant_Eulercharacteristic}

In Section~\ref {sec:Group_cohomology_and_the_equivariant_Eulercharacteristic}
we relate the group cohomology of $\Gamma$ to the $G$-Euler characteristic of
$L\backslash \underline{E}\Gamma$, where $\underline{E}\Gamma$ is the
classifying space for proper actions.

 %%%%%%%%%%%%%%%%%%%%%%%%%%%%%%%%%%%%%%%%%%%%%%%%%%%%%%%%%%%%%%%%%%%%%%%% 
\subsection*{Notations and conventions}
All our modules will be left modules. Some of our results hold in more general
situations than considered in the introduction; in such cases we will use the
letter $K$ for arbitrary finite groups, whereas $G$ is used for cyclic groups
only.

Given a chain complex $P_*$ of modules over $\IZ L$, we denote by $P_*[n]$ the
shifted chain complex given by $\bigl(P_*[n]\bigr)_i = P_{n+i}$ with
differential $\partial_{P[n]} = (-1)^n \partial_P$. A \emph{map} of two chain
complexes $f:P_*\to Q_*$ is an element of
\[ \hom^*(P,Q) = \prod_{i\in \IZ} \hom_{\IZ L} (P_{*+i},P_*), \] 
and for such an $f$ we write $df = \partial_Q
f - (-1)^n f \partial_P$. With this notation, $f$ is a \emph{chain map} if and
only if $df=0$.

Suppose we are given a group homomorphism $\rho:K\to \aut_{\IZ}(L)$; it puts the
structure of a $\IZ K$-module on $L$. For every $k\in K$, we write $\rho^k$ for
the associated automorphism of $L$, and we define $\tau^k=\IZ \rho^k$ to be the
corresponding ring automorphism of $\IZ L$. Whenever $P$ is a $\IZ L$-module, we
denote by $P^k$ the $\IZ L$-module obtained from $P$ by restricting scalars with
the ring automorphism $\tau^k$. This construction extends in an obvious way to
chain complexes of $\IZ L$-modules, leaving the differentials unchanged.

In the special case $K=G=\IZ / m\IZ$, we fix a generator $t$ of $G$ and write
$\rho = \rho^t$ and $\tau=\tau^t$ for short.

%%%%%%%%%%%%%%%%%%%%%%%%%%%%%%%%%%%%%%%%%%%%%%%%%%%%%%%%%%%%%%%%%%%%%%%%

\subsection*{Acknowledgements}
\label{subsec:Acknowledgements}

The work was financially supported by the HCM (Hausdorff Center
for Mathematics) in Bonn, and the Leibniz-Award of the second author.

%%%%%%%%%%%%%%%%%%%%%%%%%%%%%%%%%%%%%%%%%%%%%%%%%%%%%%%%%%%%%%%%%%%%%%%%%
%%%%%%%%%%%%%%% Section 1:  Proof of Theorem~\ref{the:Tate_cohomology} %%%%%%%%%%%
%%%%%%%%%%%%%%%%%%%%%%%%%%%%%%%%%%%%%%%%%%%%%%%%%%%%%%%%%%%%%%%%%%%%%%%%%

\typeout{--------  Section 1:  Proof of Theorem~\ref{the:Tate_cohomology}--------------}

\section{Proof of Theorem~\ref{the:Tate_cohomology}}
\label{sec:Proof_of_Theorem_the:Tate_cohomology}

This section is devoted to the proof of Theorem~\ref{the:Tate_cohomology}.
Its proof needs some preparation. 

\begin{lemma}\label{lem:dvr}
  Let $p$ be a prime and let $r$ be a natural number. Put $\zeta = \exp(2\pi i/p^r)$. 
  Then the ring $\IZ_{(p)}[\zeta]\cong \IZ[\zeta]_{(1-\zeta)}$ is a discrete valuation ring.
\end{lemma}
\begin{proof}
  Recall from \cite[Lemma~10.1 in Chapter~I on page~59]{Neukirch(1999)}  that the ideal
  $(1-\zeta)\IZ[\zeta]$ is a prime ideal in $\IZ[\zeta]$, and that
  $(1-\zeta)^{(p-1)p^r}=p\, \varepsilon$ for some unit
  $\varepsilon\in\IZ[\zeta]$. 
  Since $\IZ[\zeta]$ is the ring of integers in the algebraic number field $\IQ[\zeta]$,
  it is a Dedekind domain (see~\cite[Theorem~3.1 in Chapter~I on page 17 and 
  Proposition~10.2 in Chapter~I on page~60]{Neukirch(1999)}). 
  Since the localization of a Dedekind ring at one of its prime ideals is a discrete
  valuation ring (see~\cite[Theorem~9.3 on page~95]{Atiyah-McDonald(1969)}),
  it is enough to prove the isomorphism of rings $\IZ_{(p)}[\zeta]\cong \IZ[\zeta]_{(1-\zeta)}$.

  Let $K$ be the set of positive integers not divisible by $p$, and observe that
  $\IZ_{(p)}[\zeta] = K^{-1}\IZ[\zeta]$.  Under the unique ring map
  $\IZ[\zeta]\to  \IZ/p\IZ$ mapping $\zeta$ to $1$, elements of
  $K$ map to non-zero elements, and elements of $(1-\zeta)\IZ[\zeta]$ map to
  $0$; therefore, $K\cap (1-\zeta)\IZ[\zeta] = \emptyset$, so the injective map
  $\IZ[\zeta]\to  \IZ[\zeta]_{(1-\zeta)}$ induces an injective map
  $K^{-1}\IZ[\zeta]\to  \IZ[\zeta]_{(1-\zeta)}$. We want to show that
  this map is surjective.

  Consider both sides as subrings of $\IQ[\zeta]$, and let $a,b\in\IZ[\zeta]$
  with $b\notin (1-\zeta)\IZ[\zeta]$; we want to show that $\frac{a}{b}$ is in
  the image of that map. Since $\frac{a}{b}\in\IQ[\zeta]$, there is some
  positive integer $l$ with $w:=l\cdot\frac{a}{b}\in\IZ[\zeta]$. Let us write
  $l=k\cdot p^i$ with $k\in K$; then
  \[ w\cdot b = k\cdot p^i\cdot a = k\cdot (1-\zeta)^{j} \cdot a \cdot e \] for
  some integer $j\geq 0$ and some unit $e\in\IZ[\zeta]$. Since $(1-\zeta)$
  generates a prime ideal which does not contain $b$, we conclude that
  $w=(1-\zeta)^jw'$ for some $w'\in\IZ[\zeta]$ and therefore $k\cdot
  \frac{a}{b}=w' e^{-1}$, which lies in $\IZ[\zeta]$.
\end{proof}

Next  we prove the following reduction.

\begin{lemma} \label{lem:reduction_to_prime}
It suffices to prove  Theorem~\ref{the:Tate_cohomology} and Theorem~\ref{the:positive} 
in the special case, where $m = p^r$ 
for some prime number $p$ and natural number $r$ and $L = \IZ(\zeta)^k = \bigoplus_{i=1}^k \IZ(\zeta)$ 
for some natural number $k$, where $\zeta = \exp(2\pi i/r^k)$.
\end{lemma}
\begin{proof}
  Fix a prime $p$. Let $G_p$ be the $p$-Sylow subgroup of
  $G$. Obviously $G_p$ is a cyclic group of order $p^r$ for some natural
  number $r$. Let $Q$ be the quotient $(\IZ/n)/(\IZ/p)$. Obviously $Q$
  is a cyclic group of order prime to $p$, namely of order
  $m/p^r$. Hence we obtain an isomorphism
$$\widehat{H}^i(G;\Lambda^j(L))_{(p)} =  \widehat{H}^i(G_p;\Lambda^j(L))^Q.$$
This is proved at least for cohomology and $i \ge 1$
in~\cite[Theorem~10.3 on page~84]{Brown(1982)} and extends to Tate
cohomology.  Since an abelian group $A$ is trivial if and only if
$A_{(p)}$ is trivial for all primes $p$, it suffices to prove
Theorem~\ref{the:Tate_cohomology} for $G_p$ for all primes $p$. In
other words, we can assume without loss of generality $m = p^r$.

Let $t \in G$ be a generator.  Let
$T=1+t^{p^{r-1}}+t^{2p^{r-1}}+\dots+t^{(p-1)p^{r-1}}\in\IZ G$. Then
$t^{p^{r-1}}$ fixes $T\cdot x$ for each $x\in L$, so by assumption
$T\cdot x=0\in L$.  Therefore, $T\cdot L=0$, and $L$ is a 
$\IZ G /T\cdot \IZ G$-module. Now the ring epimorphism
$\pr \colon \IZ G \to \IZ[\zeta]$ sending a fixed generator $t$ of $G$ to 
$\zeta = \exp(2\pi i/p^r)$ is surjective and  contains $T$ in its kernel. Since $\IZ G / T\cdot \IZ G$
and $\IZ[\zeta]$ are finitely generated free abelian groups of the same rank,
$\pr$ induces a ring  isomorphism
\begin{eqnarray}
&\overline{\pr} \colon \IZ G /T\cdot \IZ G \xrightarrow{\cong} \IZ[\zeta].&
\label{ZG/T_and_Z[zeta]}
\end{eqnarray}
We have seen before that $\IZ[\zeta]$ is a Dedekind
domain. Every finitely generated torsion-free module over a Dedekind domain is a
direct sum of ideals 
(see~\cite[Lemma~1.5 on page~10 and remark on page~11]{Milnor(1971)}), 
so $L \cong I_1\oplus\dots\oplus I_k$ for
some ideals $I_j\subseteq \IZ[\zeta]$. Now $I_j\otimes \IZ_{(p)}$ is
an ideal in $\IZ_{(p)}[\zeta]$ which is a discrete valuation ring (see
Lemma~\ref{lem:dvr}). Since a discrete valuation ring is a principal
ideal domain with a unique maximal ideal 
(see~\cite[Proposition~9.2 on page~94]{Atiyah-McDonald(1969)}), 
$L_{(p)} \cong \bigl(\IZ_{(p)}[\zeta]\bigr)^{k}$ as modules over $\IZ_{(p)}[\zeta]$.
This implies that $\Lambda_{\IZ_{(p)}}^j (L_{(p)})$ and
$\Lambda^j_{\IZ_{(p)}}((\IZ[\zeta]_{(p)})^k)$ are isomorphic as 
$\IZ G$-modules. For any $\IZ G$-module $M$, the map $M\to  M_{(p)}$
induces an isomorphism
$\widehat{H}^*(G,M)=\widehat{H}^*(G,M)_{(p)}\cong
\widehat{H}^*(G,M_{(p)})$. Therefore,
\begin{multline*}
  \widehat{H}^*(G,\Lambda^j(L)) 
  \cong 
   \widehat{H}^*\bigl(G,(\Lambda^j (L)_{(p)}\bigr) 
   \cong 
    \widehat{H}^*\bigl(G,\Lambda_{\IZ_{(p)}}^j (L_{(p)})\bigr)  
  \\
   \cong
   \widehat{H}^*\bigl(G,\Lambda^j_{\IZ_{(p)}}((\IZ[\zeta]_{(p)})^k)\bigr)
    \cong 
   \widehat{H}^*\bigl(G,\Lambda^j(\IZ[\zeta]^k)_{(p)}\bigr) 
  \cong
   \widehat{H}^*\bigl(G,\Lambda^j(\IZ[\zeta]^k)\bigr).
 \end{multline*}
 Hence it suffices to prove Theorem~\ref{the:Tate_cohomology} 
 in the case $m = p^r$ and $L = \IZ[\zeta]^k$.

 The argument in the proof applies also to Theorem~\ref{the:positive}.
\end{proof}

Recall that a permutation $\IZ G$-module is a $\IZ G$-module of the
shape $\IZ [S]$ for some finite $G$-set $S$.
The main technical input in the proof of Theorem~\ref{the:Tate_cohomology} will be:

\begin{proposition}\label{pro:specific_resolutions}
  Suppose $m = p^r$ for some prime number $p$ and natural number $r$.
  For $j \ge 0$ there is a long exact sequence of $\IZ G$-modules
  \[ 
  0 \to  P\to  F_1\to \dots\to  F_{j}
  \to  \Lambda^{j} \IZ[\zeta] \to  0 \] 
  where $P$ is a permutation $\IZ G$-module and the $F_i$'s
  are free $\IZ G$-module 
\end{proposition}
\begin{proof}
Define the $\IZ G$-module $A=\IZ G / (1-t^{p^{r-1}}) \cdot \IZ G$. Note that we obtain 
from~\eqref{ZG/T_and_Z[zeta]} the short exact sequence of $\IZ G$-modules
\[ 0 \to  A \to  \IZ G \to  \IZ[\zeta] \to  0 \]
whose underlying sequence of free $\IZ$-modules splits. We therefore get a long
exact sequence of $\IZ$-modules
\begin{multline}\label{schurcomplex}
  0 \to  \Gamma^j A \to  \Gamma^{j-1} A \otimes \Lambda^1 \IZ G 
\to  \Gamma^{j-2} A\otimes \Lambda^2 \IZ G \to  \dots \\
  \dots\to  \Gamma^{1} A \otimes \Lambda^{j-1} \IZ G \to 
  \Lambda^j \IZ G \to  \Lambda^j \IZ[\zeta] \to  0.
\end{multline}
(see~\cite[Definition~V.1.6 and Corollary~V.1.15]{Akin-Buchsbaum-Weyman(1982)}), 
which is in fact a sequence of $\IZ G$-modules. 
Here, $\Gamma^j A$ denotes the $j$-th divided powers on $A$ (see, e.g., \cite[I.4]{Akin-Buchsbaum-Weyman(1982)}). Our aim
is to write this sequence as a direct sum of several sequences, each of which
has one of the following properties:
\begin{itemize}
\item it either does not contribute to $\Lambda^j \IZ[\zeta]$, or
\item all its middle terms are free $\IZ G$-modules.
\end{itemize}
For this we introduce a grading as follows. Define a $\IZ$-basis of $A$ by
taking $\mathcal{A}=\{[1],[t],[t^2],\dots,[t^{p^{r-1}-1}]\}$. Let $S'$ be the
additive semi-group of functions (i.e., maps of sets) $\IZ/p^{r-1}\IZ\to 
\mathbb{N}$, where $\mathbb{N}$ is the set of non-negative integers. There is a
unique way of turning $\Gamma^* A$ into an $S'$-graded ring such that the degree
of $[t^i]\in\Gamma^1 A$ is the function in $S'$ sending $[i]$ to $1$ and all
other elements to $0$. Explicitly, the degree of the $\IZ$-basis element
$a_1^{[e_1]}\dots a_m^{[e_m]}$ (with $a_i\in \mathcal{A}$ for all $i$, $e_i\geq
0$) is given by the function
\[ \IZ/p^{r-1}\IZ \ni [i] \mapsto \sum_{a_s=[t^i]} e_s. \] Similarly, there is a
unique $S'$-graded ring structure on $\Lambda^* \IZ G$ such that the degree of
$t^i\in\Lambda^1 \IZ G$ is the function in $S'$ sending $[\text{$i$ mod
  $p^{r-1}$}]$ to $1$ and all other elements to $0$. We therefore get an induced
$S'$-grading on $\Gamma^* A\otimes \Lambda^* \IZ G$ (by requiring $|a\otimes
b|=|a|+|b|$ for homogeneous $a$, $b$), which restricts to an $S'$-grading on
$\Gamma^{j-i} A\otimes \Lambda^{i} \IZ G$.

We claim that the differential of the exact sequence~\eqref{schurcomplex}
respects this grading. To verify this, note that a $\IZ$-basis element
$a_1^{[e_1]}\dots a_m^{[e_m]}\otimes b_1\wedge \dots\wedge b_i$ is mapped to
$a_1^{[e_1-1]}\dots a_m^{[e_m]}\otimes (T\cdot a_1)\wedge b_1\wedge\dots\wedge
b_i$ plus other terms of similar shape. Note here that $T\cdot a_1$ is a
well-defined element in $\IZ G$, and its a sum of elements having the same
degree in $S'$ as $a_1$.

On the other hand, the $G$-action does not quite respect the grading; in fact,
multiplication by $t$ corresponds to a shift of the degree function $\IZ /
p^{r-1}\IZ\to \mathbb{N}$. We therefore define $S = S' /
(\IZ/p^{r-1}\IZ)$, where $\IZ/p^{r-1}\IZ$ acts on the functions in $S'$ by
shifting. We get an induced $S$-grading on $\Gamma^{j-i} A\otimes \Lambda^{i}
\IZ G$, and now both the $G$-action and the differential of~\eqref{schurcomplex}
respect this grading. Therefore the exact sequence is a direct sum of exact
sequences, one for each element of $S$. For $\sigma\in S$ let us write $E_\sigma
= \bigl(\dots\bigr)_\sigma$ for the degree-$\sigma$-part of the exact sequence,
i.e.,
\[ \dots \to  (\Gamma^{j-i} A\otimes \Lambda^i \IZ G)_\sigma \to 
\dots \to  (\Lambda^j \IZ G)_\sigma \to  (\Lambda^j
\IZ[\zeta])_\sigma \to  0. \] 
The proof is now completed by applying the following 
Lemma~\ref{lem:gradings} because $\Gamma^j A$ is a permutation module.
\end{proof}

\begin{lemma}\label{lem:gradings}
  Let $\sigma\in S$ be represented by $f\in S'$. If $f(w)<p$ for all
  $w\in\IZ/p^{r-1}\IZ$, then the module $(\Gamma^{j-i}A\otimes \Lambda^i\IZ
  G)_\sigma$ is free as $\IZ G$-module for $0<i\leq j$. If $f(w)\geq p$ for some
  $w$, then $\Lambda^j(\IZ[\zeta])_\sigma = 0$.
\end{lemma}
\begin{proof}
  For the first part, it is enough to check that the action of $t^{p^{r-1}}$ on
  the canonical $\IZ$-basis elements $\beta = a_1^{[e_1]}\dots a_m^{[e_m]}\otimes
  b_1\wedge \dots\wedge b_i$ for $a_r \in \cala$, $e_r \ge 0$  
  and $b_s \in G = \IZ/p^r$ has no fixed points. 
  Suppose that $t^{p^{r-1}} \beta = \beta$. 
  Then for each $l \in \{0,1,\ldots, p-1\}$ there exists $u(l) \in \{1,2, \ldots, i\}$ 
  with $t^{lp^{r-1}} b_1 = b_{u(l)}$.    Obviously $b_{u(l)} = b_{u(l')}$ if and only if $l = l'$ 
  since $\beta \not = 0$ and $t^{lp^{r-1}} b_1 = t^{l'p^{r-1}} b_1  \Leftrightarrow l = l'$. We conclude
  where $\overline{b_1} \in \IZ/p^{r-1}$ is the reduction of $b_1 \in G = \IZ/p^r$
  \begin{multline*}
   f(\beta)(\overline{b_1}) 
   = \sum_{r=1}^m f(a_r^{[e_r]})(\overline{b_1})  + \sum_{s=1}^i f(b_s)(\overline{b_1})
  \ge  \sum_{l = 0}^{p-1} f(b_{u(l)})(\overline{b_1}) 
   \\
    = \sum_{l = 0}^{p-1} f(t^{lp^{r-1}} b_1)(\overline{b_1}) = \sum_{l = 0}^{p-1} 1 = p.
  \end{multline*}
  For the second assertion we need to check that the map $(\Lambda^j \IZ
  G)_\sigma \to  \Lambda^j \IZ[\zeta]$ is zero, so let us start with an
  element $\beta =b_1\wedge b_2\wedge\dots\wedge b_j$ in 
  $(\Lambda^j \IZ  G)_\sigma$ for $b_s\in G$. 
  Fix $w\in \IZ/p^{r-1}$ with $f(w)\geq p$. Then $p \le j$ and by possible renumbering the $b_s$-s,
  we can arrange that $b_1,b_2,\dots,b_p$ belong to the set 
  $\{t^w, t^{w+p^{r-1}}, \dots, t^{w+(p-1)p^{r-1}} \}$. Furthermore, they are
  pairwise different (otherwise $\beta=0$), so we may assume that $b_l =
  t^{w+lp^{r-1}}$ for all $l=1,2,\dots,p$. Recall that 
  $T=1+t^{p^{r-1}}+t^{2p^{r-1}}+\dots+t^{(p-1)p^{r-1}}\in\IZ G$. 
  Then $T\cdot b_1 =   b_1+b_2+\dots+b_p$, so that 
  $\beta = (T\cdot b_1)\wedge b_2\wedge\dots\wedge b_j$,
  but the latter maps to zero in $\Lambda^j \IZ[\zeta]$.
  This finishes the proof of Lemma~\ref{lem:gradings} and hence of
  Proposition~\ref{pro:specific_resolutions}.
\end{proof}

Now we can finish the proof of Theorem~\ref{the:Tate_cohomology}.
\begin{proof}[Proof of of Theorem~\ref{the:Tate_cohomology}]
By Lemma~\ref{lem:reduction_to_prime} we can  assume without loss of generality
that $m = p^r$ and $L = \IZ[\zeta]^k$ for $\zeta = \exp(2\pi i/p^r)$.
We will show by induction over $k$ that for $k \ge 1$ and $j_1, j_2, \ldots j_k \ge 0$ there exists a
long exact sequence of $\IZ G$-modules
\begin{multline}
\label{desired_long_exact_sequence}
  0 \to  P\to  F_1\to \dots\to  F_{j_1+\dots+j_k}
  \to  \Lambda^{j_1} \IZ[\zeta]  \otimes \cdots \otimes \Lambda^{j_k} \IZ[\zeta] \to  0,
\end{multline}
where $P$ is a direct summand in a permutation module.
Then Theorem~\ref{the:Tate_cohomology} follows since there is the exponential law 
\begin{eqnarray}
& \Lambda^*(X\oplus Y)\cong \Lambda^*(X)\otimes \Lambda^*(Y),  & 
\label{exponential_law}
\end{eqnarray}
Shapiro's Lemma (see~\cite[(5.2) on page 136]{Brown(1982)}) saying that for a
subgroup $H \subseteq G$ we have 
\[\widehat{H}^i(G;\IZ[G/H]) \cong \widehat{H}^i(H;\IZ),
\]
the computation 
\[ 
\widehat{H}^i(\IZ/m;\IZ) = 0 \quad \text{for}\;  i \; \text{odd},
\] 
the formula  
\[
\widehat{H}^i(G;M_1 \oplus M_2) \cong  \widehat{H}^i(G;M_1)  \oplus \widehat{H}^i(G;M_2),
\]
and the isomorphism (see~\cite[(5.1) on page 136]{Brown(1982)}) 
\[
\widehat{H}^i\bigl(G;\Lambda^{j_1} \IZ[\zeta]  \otimes \cdots \otimes \Lambda^{j_k} \IZ[\zeta] \bigr)  
\cong \widehat{H}^{i + j_1 +j_2 + \cdots + j_k}(G;P).
\]
The induction beginning $k = 1$ has already been taken care of in  Proposition~\ref{pro:specific_resolutions}.
The induction step from $k-1$ to $k \ge 2$ is done as follows.  By induction hypothesis
there exists exact sequences of $\IZ G$-modules
\[ 
  0 \to  P\to  F_1\to \dots\to  F_{j_1+\dots+j_{k-1}}
  \to  \Lambda^{j_1} \IZ[\zeta]  \otimes \cdots \otimes \Lambda^{j_{k-1}}  \IZ[\zeta] \to  0. 
\] 
and 
\[ 
  0 \to  Q\to  F'_1\to \dots\to  F'_{j_{k}}
  \to  \Lambda^{j_{k}} \IZ[\zeta] \to  0. 
\] 
where $P$ and $Q$ are permutation modules.  Since
$\Lambda^{j_1} \IZ[\zeta] \otimes \cdots \otimes \Lambda^{j_{k-1}} \IZ[\zeta]$ 
is finitely generated free as $\IZ$-module, we obtain an
exact sequence of $\IZ G$-modules
\begin{multline*}
  0 \to  \Lambda^{j_1} \IZ[\zeta]  \otimes \cdots \otimes \Lambda^{j_{k-1}}  \IZ[\zeta] \otimes Q
\to  \Lambda^{j_1} \IZ[\zeta]  \otimes \cdots \otimes \Lambda^{j_{k-1}}  \IZ[\zeta] \otimes F'_1
\\
\to \dots
\to  \Lambda^{j_1} \IZ[\zeta]  \otimes \cdots \otimes \Lambda^{j_{k-1}}  \IZ[\zeta] \otimes F'_{j_{k}}
  \to  \Lambda^{j_1} \IZ[\zeta]  \otimes \cdots \otimes \Lambda^{j_{k-1}}  \IZ[\zeta] \otimes \Lambda^{j_{k}} \IZ[\zeta]\to  0,
\end{multline*}
where all the modules except the one at the beginning and the one at
the end are finitely generated free $\IZ G$-modules.  Analogously we
we obtain an exact sequence $\IZ G$-modules
\[
0 \to  P \otimes Q \to  F_1 \otimes Q  \to \dots\to  F_{j_1+\dots+j_{k}} \otimes Q 
  \to  \Lambda^{j_1} \IZ[\zeta]  \otimes \cdots \otimes \Lambda^{j_{k-1}}  \IZ[\zeta] \otimes Q  \to  0,
\]
where all the modules except the one at the beginning and the one at
the end are finitely generated free $\IZ G$-modules and $P \otimes Q$
is a permutation module. Splicing the last two
long exact sequences together yields the desired long exact
sequence~\eqref{desired_long_exact_sequence} of $\IZ G$-modules.  This
finishes the proof of Theorem~\ref{the:Tate_cohomology}.
\end{proof}

%%%%%%%%%%%%%%%%%%%%%%%%%%%%%%%%%%%%%%%%%%%%%%%%%%%%%%%%%%%%%%%%%%%%%%%%%
%%%%%%%%%%%%%%%%%%%%%%%%% Section 2:  Tthe cohomology of $\Gamma$ %%%%%%%%%%%%%%%%%
%%%%%%%%%%%%%%%%%%%%%%%%%%%%%%%%%%%%%%%%%%%%%%%%%%%%%%%%%%%%%%%%%%%%%%%%%

\typeout{-------  Section 2: Accessing the cohomology of Gamma by underline{E}Gamma -------}

\section{The cohomology of $\Gamma$}
\label{sec:The_cohomology_of_Gamma}

In this section we present a computation of the group cohomology of
the semi-direct product $\Gamma = L \rtimes_{\phi} G$ in high degrees
provided that $G$ acts freely on $L$ outside the origin. It is
independent of the Lyndon-Hochschild-Serre spectral sequence but uses
classifying spaces for proper actions. For a survey on classifying
spaces for families and in particular for the classifying space for
proper actions we refer for instance to~\cite{Lueck(2005s)}.

Here we will only need the following facts.  A 
\emph{model $\underline{E}\Gamma$ for proper $\Gamma$-actions} is a
$\Gamma$-$CW$-complex whose isotropy groups are all finite and whose
$H$-fixed point set is contractible for every finite subgroup $H
\subseteq \Gamma$.  Such a model exists and is unique up to $\Gamma$-homotopy. We
will denote by $\underline{B}\Gamma$ the quotient $\Gamma \backslash
\underline{E}\Gamma$.

Now we are ready to prove Theorem~\ref{the:cohomology_of_Gamma}.

\begin{proof}[Proof of Theorem~\ref{the:cohomology_of_Gamma}]
Since the $G$-action on $L$ is free outside the origin,
every non-trivial finite subgroup $H \subseteq \Gamma$ is contained in
a unique maximal finite subgroup $M$ and for every maximal finite
subgroup $M \subseteq \Gamma$ we have $N_{\Gamma} M = M$
(see~\cite[Lemma~6.3]{Lueck-Stamm(2000)}. We obtain
from~\cite[Corollary~2.11]{Lueck-Weiermann(2007)} a cellular
$\Gamma$-pushout
\begin{eqnarray}
&
\xycomsquareminus{\coprod_{M \in \calm} \Gamma \times _M EM}{i_0}{E\Gamma}
  {\coprod_{M \in \calm} \pr_M }{f}
  {\coprod_{M \in \calm} \Gamma/M}{i_1}{\underline{E}\Gamma}
 &
\label{G-pushout_forunderlineEGamma}
\end{eqnarray}
where $i_0$ and $i_1$ are inclusions of $\Gamma$-$CW$-complexes, $\pr_M$ is the
obvious $\Gamma$-equivariant projection and $\calm$ is a complete system of representatives 
the set of conjugacy classes of maximal finite subgroups of $\Gamma$.  Taking the quotient with respect
to the $\Gamma$-action yields the cellular pushout
\[
  \xycomsquareminus{\coprod_{M \in \calm} BM}{j_0}{B\Gamma}
  {\coprod_{M \in \calm} \overline{\pr}_M }{\overline{f}}
  {\coprod_{M \in \calm} \pt}{j_1}{\underline{B}\Gamma}
\]
where $j_0$ and $j_1$ are inclusions of $CW$-complexes, $\overline{\pr}_M$ is the
obvious projection.  It yields the following long exact sequence for $k \ge 0$
\begin{multline}
  0 \to {H}^{2k}(\underline{B}\Gamma) \xrightarrow{\overline{f}^*} {H}^{2k}(\Gamma)
  \xrightarrow{\varphi^{2k}} \bigoplus_{(M) \in \calm} \widetilde{H}^{2k}(M)
  \\
  \xrightarrow{\delta^{2k}} {H}^{2k+1}(\underline{B}\Gamma)
  \xrightarrow{\overline{f}^*} {H}^{2k+1}(\Gamma) \to 0
  \label{long_exact_cohomology_sequences_for_bub(Gamma)_BGamma)}
\end{multline}
where $\varphi^*$ is the map induced by the various inclusions $M \subset \Gamma$ for $M \in \calm$
and $\widetilde{H}^{2k}(M)$ is reduced cohomology, i.e., $\widetilde{H}^{2k}(M) = H^{2k}(M)$ for $k \ge 1$
and $\widetilde{H}^{0}(M) = 0$.

One can construct a model for $\underline{E}\Gamma$ whose dimension as a
$\Gamma$-$CW$-complex is $n$ (see~\cite[Example~5.26]{Lueck(2005s)}). Namely,
one can take $\IR \otimes_{\IZ} L$ with the obvious $\Gamma$-action coming from
the $L$-action given by translation and the $G$-action coming from $G
\xrightarrow{\rho} \aut_{\IZ}(L) \to \aut_{\IR}(\IR \otimes_{\IZ} L)$, where the
second map comes from applying $\IR \otimes_{\IZ} -$.  Now
Theorem~\ref{the:cohomology_of_Gamma} follows.
\end{proof}

%%%%%%%%%%%%%%%%%%%%%%%%%%%%%%%%%%%%%%%%%%%%%%%%%%%%%%%%%%%%%%%%%%%%%%%%%
%%Section 3:  The Relation of Conjecture~\ref{con:Adem-Ge-Pan-Petrosyan}  and Theorem~\ref{the:Tate_cohomology} Proof of Theorem~ref{the:positive} %%%%%%
%%%%%%%%%%%%%%%%%%%%%%%%%%%%%%%%%%%%%%%%%%%%%%%%%%%%%%%%%%%%%%%%%%%%%%%%%

\typeout{--------  Section 3:  The Relation of Conjecture~ref{con:Adem-Ge-Pan-Petrosyan}  and Theorem~ref{the:Tate_cohomology} --------------}

\section{The Relation of Conjecture~\ref{con:Adem-Ge-Pan-Petrosyan} and Theorem~\ref{the:Tate_cohomology}}
\label{sec:The_Relation_of_Conjecture_and_Theorem_Tate}

\begin{lemma}
  \label{lem:Theorem_Tate_implies_Theorem_positiv_and_vice_versa}
  Suppose that $G$ acts freely outside the origin on $L$. Then
  Theorem~\ref{the:Tate_cohomology} is true if and only if the differentials in the 
  Lyndon-Hochschild-Serre spectral sequence associated to $\Gamma = G \rtimes_{\phi} G$ vanish.
\end{lemma}
\begin{proof}
  The cup-product induces isomorphisms $\Lambda^j H^1(L)
  \xrightarrow{\cong} H^j(L)$, natural with respect to automorphisms
  of $L$, for $j \ge 0$. By the universal coefficient theorem we obtain
  an isomorphism $L\dual := \hom_{\IZ}(L,\IZ) \xrightarrow{\cong}
  H^1(L)$ which is natural with respect to automorphisms of $L$.
  Putting this together we obtain an isomorphism, natural with respect
  to automorphisms of $L$,
  \begin{eqnarray*}
    \Lambda^j L\dual & \xrightarrow{\cong} & H^j(L).
  \end{eqnarray*}

  We first show that Theorem~\ref{the:Tate_cohomology} implies
  the vanishing of all the differentials. From Theorem~\ref{the:Tate_cohomology}
  we conclude that $E^2_{i,j} = 0$ for $i + j$ odd and $i \ge 1$ since
  the Tate cohomology in dimensions $i \ge 1$ coincides with
  cohomology. Hence by the checkerboard pattern of the $E^2$-term the
  only non-trivial differentials which can occur are those who start
  at the vertical axis or end at the horizontal axis at a point in odd position on the axis.
  To show that all these differentials  vanish, we have to
  prove that all edge homomorphisms are trivial. This boils down to
  show of the projection $\pr \colon \Gamma \to G$ and for the
  inclusion $i\colon L \to \Gamma$ that the map 
 $\pr^* \colon H^r(G)  \to H^r(\Gamma)$ is injective  and the map 
 $i^* \colon  H^r(\Gamma) \to H^r(L)^G$ is surjective for odd $r$.  The map
 $\pr^*$ is injective since $\pr$ has a section. Let 
 $i^! \colon    H^r(L) \to H^r(\Gamma)$ be the transfer map.  Its composition with
 $i^* \colon H^r(\Gamma) \to H^r(L)^G$ is the map $H^r(L) \to
  H^r(L)^{G}$ given by multiplication with the norm element $N := \sum_{g \in G} g$,
  and the cokernel of this map is isomorphic to $\widehat{H}^0(G;H^r(L))$ 
 (see \cite[(5.1) on page~134]{Brown(1982)}). By assumption  $\widehat{H}^0(G;H^r(L))$
 vanishes for odd $r$. Hence $i^* \circ {i^!}$ is surjective in odd dimensions
 and hence $i^*$ is surjective in odd dimensions. This finishes
the proof that all differentials in the Leray-Hochschild-Serre spectral sequence vanish.

Now suppose that all differentials vanish. Then
Theorem~\ref{the:positive} holds by the following argument. We know
that $H^{2m+1}(\Gamma)$ vanishes for $2m+1 > n$ by
Theorem~\ref{the:cohomology_of_Gamma}. Since all differentials in the
Leray-Serre spectral sequence vanish, we conclude that $\widehat{H}^i(G;H^j(L)) = H^i(G;H^j(L)) =0$ 
for $i \ge 1$, $i +j$ odd and $i + j > n$. Since the Tate cohomology is
$2$-periodic for finite cyclic groups, this implies that
$\widehat{H}^i(G;H^j(L)) = 0$ holds for all $i,j$ with $i +j$ odd.
\end{proof}

\begin{remark}
Notice that Theorem~\ref{the:positive}
and Lemma~\ref{lem:Theorem_Tate_implies_Theorem_positiv_and_vice_versa}
give another proof of Theorem~\ref{the:Tate_cohomology}
independent of the one presented in Section~\ref{sec:Proof_of_Theorem_the:Tate_cohomology}.
\end{remark}

%%%%%%%%%%%%%%%%%%%%%%%%%%%%%%%%%%%%%%%%%%%%%%%%%%%%%%%%%%%%%%%%%%%%%%%%% 
%%%%%%%%%%%%%%%%%% Section 4:  The cohomology classes $[\alpha_s]$ %%%%%%%%%%%%%%%%%%%%%
%%%%%%%%%%%%%%%%%%%%%%%%%%%%%%%%%%%%%%%%%%%%%%%%%%%%%%%%%%%%%%%%%%%%%%%%%

\typeout{--------  Section 4:  The cohomology classes $[\alpha_s]$  --------------}

\section{The cohomology classes $[\alpha_s]$}
\label{sec:The_cohomology_classes_[alpha_s]}

Next  we introduce certain cohomology classes which will be used to describe the second differential
in the Lyndon-Hochschild-Serre spectral sequence and are obstructions to the existence of a
compatible group action in the sense of~\cite[Definition~2.1]{Adem-Pan(2006)}). We will also give a description
in terms of endomorphism of free groups.

%%%%%%%%%%%%%%%%%%%%%%%%%%%%%%%%%%%%%%%%%%%%%%%%%%%%%%%%%%%%%%%%%%%%%%%% 

\subsection{The definition of the classes $[\alpha_s]$}
\label{subsec:The_definition_of_the_classes_[alpha_s]}

Let $\rho \colon L \to L$ be the automorphism of $L$ given by multiplication
with a fixed generator $t$ of the cyclic group $G$ of order $m$.  Denote by 
$\tau = \IZ \rho \colon \IZ G \to \IZ G$ the ring automorphism of $\IZ L$ induced by
$\rho$.  Obviously $\rho^m = \id_L$ and $\tau^m = \id_{\IZ L}$.

Let $(P_*,\partial)$ be a projective resolution over $\IZ L$ of the trivial
module $\IZ$ with the additional property that the complex $\hom_{\IZ L}(P_*,\IZ)$ 
has trivial differentials. As usual, let $\tau^*$ denote the
endofunctor of the category of $\IZ L$-modules given by pulling back scalars
along $\tau$. Then $P_*$ and $\tau^* P_*$ both are projective $\IZ L$-resolutions of the
trivial module $\IZ$, so there is a chain map $z \colon \tau^* P_*\to P_*$
lifting the identity of $\IZ$. Then 
$H^*(L)=\hom_{\IZ L}(P_*,\IZ)=\hom_{\IZ  L}(\tau^* P_*,\IZ)$ 
gets a $\IZ G$-module structure via the map $\hom_{\IZ L}(z,\id_{\IZ})$. 
Note that $H^*(L)\cong \Lambda^* L\dual$ as $\IZ G$-modules.

Now the $m$-fold composition $z^m$ is a $\IZ L$-chain map $P_*\to P_*$
lifting the identity of $\IZ$. Therefore there is a $\IZ L$-chain homotopy
$y \colon P_*\to  P_*[1]$ with $dy=\partial y+y\partial=z^m-1$.  This induces a map
\[
\alpha_s\colon H^{s+1}(L) \cong \hom_{\IZ L}(P_{s+1},\IZ) 
\xrightarrow{\hom_{\IZ L}(y,\id_{\IZ})} \hom_{\IZ L}(P_s,\IZ) \cong H^s(L). 
\] 
We claim that $\alpha_s$ is a $G$-equivariant map. To see this, consider the map 
$zy-yz \colon \tau^* P_*\to  P_*[1]$. Since
\[ 
d(zy-yz) = z(dy) - (dy)z = z(z^m-1) - (z^m-1)z = 0,
\] 
it is a chain map. Therefore, it must be null-homotopic, so there is a map 
$x \colon \tau^* P_*\to P_*[2]$ with $dx=\partial x-x\partial=zy-yz$. 
Since $\hom_{\IZ L}(P_*,\IZ)$ has trivial differential, we get $0=\hom_{\IZ L}(dx,\id_{\IZ})=\hom_{\IZ L}(zy-yz,\id_{\IZ})$,
proving that $\alpha_s$ is indeed $\IZ G$-linear.

We can think of $\alpha_s$ as an element 
\[
\alpha_s \in \Ext^0_{\IZ G}(H^{s+1}(L),H^s(L)) = \hom_{\IZ G}(H^{s+1}(L),H^s(L)).
\]
Using the obvious pairing coming from the tensor product over $\IZ$ (with the diagonal $G$-action)
\[
\Ext^i_{\IZ G}(M_1,M_2) \otimes \Ext^j_{\IZ G}(N_1,N_2) \to \Ext^{i+j}_{\IZ G}(M_1 \otimes_{\IZ} N_1,M_2\otimes_{\IZ} N_2)
\]
and the generator of the group $\Ext^2_{\IZ G}(\IZ,\IZ) = H^2(G) \cong \IZ/m$
given by the extension $\IZ \to \IZ G \xrightarrow{1-t} \IZ G \to \IZ$ for the fixed generator $t \in G$,
we obtain the desired class
\begin{eqnarray}
& [\alpha_s]  \in \Ext^2_{\IZ G}(H^{s+1}(L),H^s(L)). &
\label{[alpha_s]}
\end{eqnarray}

It is not hard to check that the classes $[\alpha_s]$ are independent
of the choices of $P_*$, $z$, and $y$.

%%%%%%%%%%%%%%%%%%%%%%%%%%%%%%%%%%%%%%%%%%%%%%%%%%%%%%%%%%%%%%%%%%%%%%%% 

\subsection{The second differential}
\label{subsec:The_second_differential}

Notice that we have for $\IZ G$-modules $M_1$ and $M_2$ a pairing
\[
\Ext^{i}_{\IZ G}(\IZ,M_1) \otimes \Ext^{j}_{\IZ G}(\IZ,\hom_{\IZ}(M_1,M_2)) \to \Ext^{i+j}_{\IZ G}(\IZ,M_2)
\]
coming from the cup product with respect to the pairing
$M_1 \otimes_{\IZ} \hom_{\IZ}(M_1,M_2) \to M_2$ sending $m_1 \otimes f$ to $f(m_1)$.

\begin{lemma}\label{lem:differential_is_product_with_alpha}
The map 
\[
H^r(G,H^{s+1}L) \xrightarrow{- \cdot [\alpha_s]\cdot (-1)^{r} }  H^{r+2}(G,H^*L) 
\]
given by taking products:
\begin{align*}
 H^r(G,H^{s+1}L) = \Ext_{\IZ G}^r(\IZ,H^{s+1}L) &\to \Ext_{\IZ G}^{r+2}(\IZ,H^sL) = H^{r+2}(G,H^sL) \\ 
 u &\mapsto u\cdot  [\alpha_s]\cdot (-1)^{r} 
\end{align*}
is the $d_2$-differential of the  Lyndon-Hochschild-Serre spectral sequence.
\end{lemma}
  Let us remark here that this is shown in a slightly different setup in
  \cite[Corollary 2]{Siegel(1995)}. For convenience of 
  the reader, we give a proof here which is adapted to our situation.

  To do so, we use the results of~\cite{Charlap-Vasquez(1969)}. 
  Define a \emph{$G$-system} (see~\cite[Definition in~I.1 on page~534]{Charlap-Vasquez(1969)})
  to consist of maps $A(g)\in \hom_{\IZ L}(P_*,P_*^g)$ for
  every $g\in G$ and $U(g,h)\in\hom_{\IZ L}(P_*,P_*^{gh})$ for every pair
  $(g,h)\in G\times G$, such that the following conditions hold:
  \begin{align*}
    \epsilon A(g) &= \epsilon,  && \text{where $\epsilon$ is the augmentation of $P_*$, } \\
    dA(g) &=0 && \text{for all $g\in G$,} \\
    dU(g,h)&=A(gh)-A(g)A(h) && \text{for all $g,h\in G$.}
  \end{align*}
  In our case, we can define $A(t^i)=z^{m-i}$ for $i=1,\dots,m-1$ and $A(1)=1$, and we put
  \[ 
  U(t^i,t^j)=\begin{cases} 0 & \text{if $i+j>m$ or $i=0$ or $j=0$; } 
   \\ 
   -yz^{i+j-m} &
    \text{otherwise.} \end{cases} 
  \] 
  In~\cite[II.2]{Charlap-Vasquez(1969)}, characteristic classes are constructed as follows. 
  By the universal coefficient theorem, $H^n(L,X)\cong\hom_{\IZ}(H_n(L),X)$ for all $\IZ$-free
  modules $X$ with trivial $L$-action. Choose a cocycle $f^n\in\hom_{\IZ L}(P_n,H_n(L))$
  (where $H_n(L)$ is regarded as module with trivial $L$-action) representing the identity map
  in $\hom_{\IZ}(H_n(L),H_n(L))$. For each $g\in G$, there is some $F_g^{n-1}\in\hom_{\IZ L}(P_{n-1},H_n(L))$
  with $f^n\circ A_n(g)-f^n = F_g^{n-1}\partial_n$. In our case, the differential on $\hom_{\IZ L}(P_*,X)$
  is zero for every $\IZ$-free module $X$ with trivial $L$-action, so we can assume that $F_g=0$.
  
  Now~\cite[Equation~(2) in~I.2 on page~536]{Charlap-Vasquez(1969)}
  reduces to the definition 
  \[ u^n(g,h) = (gh)_*[ f^n U_{n-1} (h^{-1},g^{-1})] \in\hom_{\IZ L}(P_{n-1},H_n(L)) \]
  for all $g,h\in G$, where $(gh)_*$ is the action of $gh$ on homology $H_n(L)$, so that 
  $u^n(g,h)$ equals the composition
  \begin{align}\label{mludefinition}
     P_{n-1}^{gh} \xrightarrow{U_{n-1}(h^{-1},g^{-1})} P_n \xrightarrow{f^n} H_n(L) 
        \xrightarrow{gh} H_n(L). 
  \end{align}
  Then $u^n(g,h)$ is a cocycle defining a cohomology class $w^n(g,h)\in H^{n-1}(L,H_n(L))$.
  By \cite[Theorem~1 in~II.2.1 on page~537]{Charlap-Vasquez(1969)} the collection of these cohomology classes
  defines a class $v^n\in H^2(G,H^{n-1}(L,H_n(L)))$.
  We would like to compare this class with our $[\alpha_n]$. To do so, note that by the universal
  coefficient theorem, $H^n(L)\cong \hom_{\IZ}(H_n(L),\IZ)$, and let us denote by $D$ the isomorphism
  \[ D \colon \hom_{\IZ}(H^n(L), H^{n-1}(L))\to \hom_{\IZ}(H_{n-1}(L),H_n(L)) \]
  given by dualizing. Also, the universal coefficient theorem gives us an isomorphism
  $\hom_{\IZ}(H_{n-1}(L),H_n(L))\cong H^{n-1}(L,H_n(L))$.
  \begin{lemma}\label{lem:alpha_and_v_agree}
    When we apply $H^2(G,-)$ to the $G$-linear isomorphism 
    \begin{align*}
      \gamma \colon \hom_{\IZ}(H^n(L), H^{n-1}(L))\xrightarrow{D} \hom_{\IZ}(H_{n-1}(L),H_n(L))
                  \xrightarrow{\cong} H^{n-1}(L,H_n(L)),
    \end{align*}
    then under the resulting map the class $[\alpha_n]$ maps to the class $v^n$.
  \end{lemma}
  \begin{proof}
  Let $B_*$ be the bar resolution, whose modules are given by $B_s=(\IZ G)^{\otimes(s+1)}$ with 
  $G$ acting on the first factor. Then $B_s$ is $\IZ G$-free on generators $(g_1,\dots,g_s)$ with $g_i\in G$, and
  the differential is given by
  \[
    (g_1,\dots,g_s) \mapsto g_1(g_2,\dots,g_s) + \sum_{r=1}^{s-1} (-1)^r (g_1,\dots,g_rg_{r+1},\dots,g_s)
                            + (-1)^s(g_1,\dots,g_{s-1}).
  \]
  Let $F_*$ be the standard free resolution
  \[  \dots \to \IZ G \xrightarrow{1-t} \IZ G \xrightarrow{t^{m-1}+\dots+t+1} \IZ G 
         \xrightarrow{1-t} \IZ G. \]
  Then $v^n$ is represented by $u^n\in\hom_{\IZ G}(B_2,H^{n-1}(L,H_n(L)))$, and the first step will
  be to find a representative for $v^n$ in $\hom_{\IZ G}(F_2,H^{n-1}(L,H_n(L)))$. Note that we can
  construct a map of augmented chain complexes $F_*\to B_*$ as follows:
  \[ 
     \xymatrix@C=40pt{
         \dots \ar[r] & F_2 \ar[r]^{t^{m-1}+\dots+t+1} \ar[d]_{1\mapsto -\sum_{i=0}^{m-1}(t^i,t)} &
                        F_1 \ar[r]^{1-t} \ar[d]_{1\mapsto -(t)} & 
                        F_0 \ar[d]_{1\mapsto ()} \ar[r] & 
                        \IZ \ar@{=}[d] \\
         \dots \ar[r] & B_2 \ar[r] & B_1 \ar[r] & B_0 \ar[r] & \IZ
     }
  \]
  We therefore find a cocycle in $\hom_{\IZ G}(F_2,H^{n-1}(L,H_n(L)))$ by evaluating $u^n$ at
  $-\sum_{i=0}^{m-1} (t^i,t)$. From~\eqref{mludefinition} we see that $u^n$ is zero at most of the 
  terms (note the definition of $U_n$), the only remaining part is 
  \[ -u^n(t^{-1},t) = -[f^n U_{n-1}(t^{m-1},t)] = [f^n y]. \]
  This means that the $G$-linear map $1\mapsto [f^n y]\in H^{n-1}(L,H_n(L))$ is a cocycle in 
  $\hom_{\IZ G}(F_2,H^{n-1}(L,H_n(L)))$ representing $v^n$. 

  Recall the isomorphism $H^{n-1}(L,H_n(L))\xrightarrow{\cong} \hom_{\IZ}(H_{n-1}(L),H_n(L))$: 
  Given any cocycle $c\in \hom_{\IZ L}(P_{n-1},H_n(L))$, we can form 
  \[ \id_{\IZ} \otimes_{\IZ L} c \colon \IZ  \otimes_{\IZ L} P_{n-1} \to \IZ \otimes_{\IZ L} H_n(L)\cong H_n(L).\] 
  Passing to the homology of the complex $\IZ\otimes_{\IZ L} P_*$ yields a map $H_{n-1}(L)\to H_n(L)$,
  the image of the class $[c]$ in $\hom_{\IZ}(H_{n-1}(L),H_n(L))$. We have the natural isomorphism
  \begin{align}\label{mlzzlisomorphism}
    \IZ \otimes_{\IZ L} -  \cong \hom_{\IZ}(\hom_{\IZ L}(-,\IZ),\IZ)
  \end{align}
  for all the modules we are interested in, so the differential of $\IZ \otimes_{\IZ L} P_*$ is zero. 
  Therefore, the class $[f^n y]$ corresponds to the composition
  \[ H_{n-1}(L) \xrightarrow{\id_{\IZ}\otimes_{\IZ L} y} H_n(L) \xrightarrow{\id_{\IZ} \otimes_{\IZ L} f^n} H_n(L), \]
  where the second map is the identity by definition of $f^n$. Therefore, the $G$-linear map
  $F_2\to \hom_{\IZ}(H_{n-1}(L),H_n(L)), 1\mapsto \id_{\IZ}\otimes_{\IZ L} y$ represents the class $v^n$.
  But by~\eqref{mlzzlisomorphism} we have $D(\alpha_n)=D(\hom_{\IZ L}(y,\id_{\IZ})) = \id_{\IZ} \otimes_{\IZ L} y$.
  \end{proof}

  We continue with the proof of Lemma~\ref{lem:differential_is_product_with_alpha}.
  By~\cite[Proposition~2 and Theorem~4 in~I.3 on page~539 and~540]{Charlap-Vasquez(1969)}, 
  the $d_2$-differential can be 
  described as follows. The $G$-linear isomorphism $\vartheta \colon H^q(L)\xrightarrow{\cong} H^0(L,H^q(L))$ induces a map
  $\theta \colon H^p(G,H^q(L))\xrightarrow{\cong} H^p(G,H^0(L,H^q(L)))$. On the other hand, we have the class
  $v^q\in H^2(G,H^{q-1}(L,H_q(L)))$. The pairing $H^q(L)\otimes H_q(L)\to \IZ$ induces a product
  $H^0(L,H^q(L))\otimes H^{q-1}(L,H_q(L))\to H^{q-1}(L)$. Then for a class
  $\chi\in E_2^{p,q} = H^p(G,H^q(L))$ we have
  \[ d_2(\chi) = (-1)^p \theta(\chi) \cup v^q. \]
  To finish the proof of Lemma~\ref{lem:differential_is_product_with_alpha} it is therefore enough to show that the diagram
  \begin{align}\label{commtoprove}
    \xymatrix{H^q(L) \otimes \hom_{\IZ} (H^q(L), H^{q-1}(L))  \ar[r]^-{\text{ev}} \ar[d]_{\vartheta\otimes \gamma} 
           & H^{q-1}(L) \ar@{=}[d]  \\
       H^0(L,H^q(L)) \otimes H^{q-1}(L,H_q(L)) \ar[r] & H^{q-1}(L) 
      }
  \end{align}
  commutes, where $\gamma$ is as in Lemma~\ref{lem:alpha_and_v_agree}.
  To see this, let $X$, $Y$ and $Z$ be finitely generated free $\IZ$-modules with trivial $L$-action, and let a 
  map $Z\to Y\dual$ be given. Then we have
  a commutative diagram
  \begin{align}\label{commdiagdualize}
    \xymatrix{X\otimes \hom_{\IZ} (X,Z) \ar[r]\ar[d] & Z \ar[dd] \\
        X\otimes \hom_{\IZ} (Z\dual,X\dual) \ar[d]  \\
        X\otimes \hom_{\IZ} (Y,X\dual) \ar[r] & Y\dual 
    }
  \end{align}
  where the left-hand side vertical maps are given by dualizing and the natural map $Y\to (Y\dual)\dual$; the horizontal 
  arrows are evaluation maps. We also have the natural diagram
  \begin{align}\label{commcupproduct}
    \xymatrix{X\otimes \hom_{\IZ} (H_{q-1}(L),Y) \ar[r]\ar[d]^{\cong} & \hom_{\IZ} (H_{q-1}(L),X\otimes Y) \ar[d]^{\cong} \\
       H^0(L,X)\otimes H^{q-1}(L,Y) \ar[r]_-{\cup} & H^{q-1}(L,X\otimes Y)
    }
  \end{align}
  This commutes because by naturality, it is enough to consider the case $X=Y=\IZ$ which is a tautology. Now we put
  things together and obtain
  \[
    \xymatrix@C=10pt{
       H^q(L)\otimes [H^q(L),H^{q-1}L] \ar[rr]^{\text{ev}} \ar[d] \ar@{}[rrd]|{(a)} & & H^{q-1}(L) \ar[d] \\
       H^q(L)\otimes [H_{q-1}(L),H^q(L)\dual] \ar[rr] \ar@{=}[d] \ar@{}[rrd]|{(b)} & & [H_{q-1}(L),\IZ] \ar@{=}[d] \\
       H^q(L)\otimes [H_{q-1}(L),H^q(L)\dual] \ar[r] \ar[d] \ar@{}[dr]|{(c)} & [H_{q-1}(L),H^q(L)\otimes H^q(L)\dual] \ar[r]^-{\text{ev}_*} \ar[d] \ar@{}[dr]|{(d)}
             & [H_{q-1}(L),\IZ] \ar[d] \\
       H^0(L,H^q(L)) \otimes H^{q-1}(L,H^q(L)\dual) \ar[r]^-{\cup} & H^{q-1}(L,H^q(L)\otimes H^q(L)\dual) \ar[r]^-{\text{ev}_*} 
             & H^{q-1}(L)
    }
  \]
  where we wrote $[X,Y]$ for $\hom_{\IZ}(X,Y)$. The square (a) is~\eqref{commdiagdualize} for $X=H^q(L)$, $Y=H_{q-1}(L)$, 
  $Z=H^{q-1}(L)$, and $Z\to Y\dual$ is the map from the universal coefficient theorem. The square (b) commutes
  by definition, (d) commutes by naturality of the universal coefficient theorem, and (c) is~\eqref{commcupproduct} for
  the case $X=H^q(L)$ and $Y=H^q(L)\dual$. A quick check asserts that the ``outer square'' agrees with~\eqref{commtoprove},
  up to another application of the universal coefficient theorem $H^q(L)\dual \cong H_q(L)$. This finishes the proof
  of Lemma~\ref{lem:differential_is_product_with_alpha}.

\begin{remark}
Note that the second differential can also be identified with the composite
\[
H^r(G,H^{s+1}L) \xrightarrow{H^r(G,\alpha_s)} H^r(G,H^sL) \to  H^{r+2}(G,H^sL),
\]
where the second map is the map coming from taking the cup product
with the  generator of the group $H^2_{\IZ G}(G,\IZ)$, which is the two-periodicity isomorphism
if $r \ge 1$. Since we have first fixed a choice of generator  of $G$ and then chosen the generator
$H^2_{\IZ G}(G,\IZ)$ accordingly, the map above is indeed independent of the choice
of generator of $G$.
\end{remark}

%%%%%%%%%%%%%%%%%%%%%%%%%%%%%%%%%%%%%%%%%%%%%%%%%%%%%%%%%%%%%%%%%%%%%%%% 

\subsection{Obstructions for the existence of a compatible group action}
\label{subsec:Obstructions_for_the_existence_of_a_compatible_group_action}

These classes serve as obstructions for the existence of a compatible action in the sense
of~\cite[Definition~2.1]{Adem-Ge-Pan-Petrosyan(2008)}. 
Let us recall their definition here.
\begin{definition}[Compatible group action]\label{def:compatible_group_action}
  Let $K$ be an arbitrary group acting $\IZ$-linearly on the abelian group $A$. Suppose that $P_*\to \IZ$
  is a free resolution of the trivial $\IZ A$-module $\IZ$ over $\IZ A$. Then we say that $P_*$ admits a
  compatible $K$-action if there is an augmentation-preserving chain map 
  $t_k\colon P_*\to P_*$ for each $k\in K$ such that the following conditions
  are satisfied:
  \begin{itemize}
    \item[(1)] $t_k(p\cdot a) = a^k\cdot t_k(p)$ for all $a\in A$, $k\in K$, $p\in P_*$, where
                      $a^k$ denotes the action of $k$ on $A$,
    \item[(2)] $t_kt_{k'} = t_{kk'}$ for all $k,k'\in K$,
    \item[(3)] $t_1=1_{P_*}$.
  \end{itemize}
\end{definition}
Notice that (1) is equivalent to saying that $t_k\in \hom_{\IZ L}(P_*,P_*^k)$.

Now let $G$ be the cyclic group of order $m$ as before. The following lemma 
is an immediate consequence of the definitions and one should think of a compatible
action just as described below. A free resolution $P_*$ is called \emph{special} if
the differential of the complex $\hom_{\IZ L}(P_*,\IZ)$ is zero.

\begin{lemma}
  There is a compatible action of $G$ on a special free resolution $P_*$ if and only if the
  chain map $z$ can be chosen in such a way that $z^m=1$.  If this is
  the case, then all the Ext-classes $[\alpha_s]$ for $s\geq 0$ are zero.
\end{lemma}

Let $\{e_1,e_2,\ldots,e_n\}$ be a basis of the abelian a group
$L$. We denote by $(l)$ (with $l\in L$) the $\IZ$-basis elements of
$\IZ L$; in particular, $(0)$ is the unit of the ring $\IZ L$. From
now on, we will have a particular projective $\IZ L$ resolution $P_*$ of the trivial $\IZ L$-module $\IZ$
in mind, namely, the \emph{Koszul complex}, which is defined as follows. As
$\IZ L$-module, $P_i$ is free of rank $\binom{n}{i}$ with generators
$[i_1,i_2,\dots,i_i]$, $1\leq i_1<i_2<\dots<i_i\leq n$. The
differential $\partial$ is given by
\[ [i_1,i_2,\dots,i_i] \mapsto \sum_{j=1}^i (-1)^{j-1}
\bigl((0)-(e_j)\bigr)\cdot [i_1,\dots,\widehat{i_j},\dots,i_i]. \]

\begin{lemma}\label{lem:main_example}
  Let $m=4$, $n=3$, and let $\rho$ be given by the integral matrix
  $\left(\begin{smallmatrix} 0 & 1 & 0 \\ -1 & 0 & 1 \\ 0 & 0 & 1 \end{smallmatrix}\right)$. 

  Then $[\alpha_1]\neq 0$, and so
  there is no compatible $G$-action in this case. Nevertheless, the
  associated Lyndon-Hochschild-Serre spectral sequence collapses.
\end{lemma}
\begin{proof}
  We start by writing down the beginning of an explicit choice of
  chain map $z \colon \tau^* P_*\to  P_*$. Let 
  $z \colon \tau^*  P_0\to  P_0$ be the map $\tau^{-1}$, and define
  \begin{align*}
    z\colon \tau^* P_1&\to  P_1 \\
    [1] &\mapsto [2]  \\
    [2] &\mapsto -(-e_1)[1] \\
    [3] &\mapsto [1]+(e_1)[3].
  \end{align*}
  Let us determine $z^4\colon P_1\to  P_1$. We have
  \[ [1] \xrightarrow{z} [2] \xrightarrow{z} -(-e_1)[1]
  \xrightarrow{z} -(-e_2)[2] \xrightarrow{z} [1]. \] 
  Notice here, for instance, that 
  $z(-(-e_2)[2]) = -(\rho^{-1}(-e_2))\cdot z([2]) =  (e_1)(-e_1)[1] = [1]$.
  From the computation we get $z^4([1])=[1]$
  and $z^4([2])=[2]$. Now,
  \[ [3] \xrightarrow{z} [1]+(e_1)[3] \xrightarrow{z}
  [2]+(e_2)([1]+(e_1)[3]) = (e_2)[1] + [2] + (e_1+e_2)[3]. \]
  Therefore, $z^4$ maps $[3]$ to
  \begin{align*}
    -(-e_1-e_2)[1]  -(-e_2)[2] + (-e_1-e_2)\bigl((e_2)[1]+[2]+(e_1+e_2)[3]\bigr) \\
    = \bigl( (-e_1) -(-e_1-e_2) \bigr)\cdot[1] + \bigl(
    -(-e_2) + (-e_1-e_2) \bigr)\cdot[2] + [3]
  \end{align*}
  Now we start choosing $y$. Let $y\colon P_0\to  P_1$ be the zero
  map. Furthermore, we can put $y([1])=y([2])=0$. For $y([3])$, we
  have to choose a lift of $z^4([3])-[3]$ along $\partial$; one such
  lift is
  \[ y([3]) = (-e_1-e_2)\cdot [1,2] \] Now we will show that
  $[\alpha_1]\in\Ext_{\IZ G}^2(\Lambda^2 L\dual, L\dual) \cong
  \Ext_{\IZ G}^2(\IZ,\hom_\IZ(\Lambda^2 L\dual,L\dual))$ 
  is   non-zero. For any two $\IZ G$-modules $U,V$, we have a natural pairing
  \[
  \hom_{\IZ}(U,V) \otimes (U\otimes V\dual)\to  \IZ
  \] 
  given by $f\otimes(u\otimes v)\mapsto v(f(u))$. Consider the exterior cup
  product followed by that map:
  \[ \Ext^2_{\IZ G}(\IZ,\hom_{\IZ}(U,V)) \otimes 
  \hom_{\IZ G}(\IZ,U\otimes V\dual) \xrightarrow{\cup} \Ext^2_{\IZ G}(\IZ,\IZ)
  \cong \IZ/4\IZ. \] Now put $U=\Lambda^2 L\dual$ and $V=L\dual$, and
  denote by $e_1\dual, e_2\dual, e_3\dual$ the dual basis for
  $e_1,e_2,e_3$. Then
  \[ 
  a = (e_1\dual\wedge e_2\dual)\otimes (e_1+e_2+2e_3) -
  (e_1\dual\wedge e_3\dual)\otimes e_3 + (e_2\dual\wedge
  e_3\dual)\otimes (e_2+e_3) 
  \] is a $G$-invariant element of
  $U\otimes V\dual$. Under the  pairing $\hom_{\IZ}(U,V) \otimes (U\otimes V\dual)\to  \IZ$ 
  mentioned above   we get $\alpha_1 \otimes a = 2$. This implies
  $[\alpha_1]\cup a = 2\in \IZ/4\IZ$, and hence   $[\alpha_1]\neq 0$.

  The collapse of the spectral sequence was noted
  in~\cite[page~350]{Adem-Ge-Pan-Petrosyan(2008)}. 
\end{proof}

\begin{remark} \label{rem:mininmal_dimension} If $n \le 2$, then there always
  exists a compatible group action on the Koszul
  resolution by~\cite[Theorem~3.1]{Adem-Pan(2006)}.  Hence
  our example for a lattice without compatible group action on the Koszul
  resolution appearing in Theorem~\ref{lem:main_example} has minimal rank, namely
  $n = 3$.
\end{remark}

%%%%%%%%%%%%%%%%%%%%%%%%%%%%%%%%%%%%%%%%%%%%%%%%%%%%%%%%%%%%%%%%%%%%%%%% 

\subsection{An approach via free groups}
\label{subsec:An_approach_via_free_groups}

Let us establish a connection to free groups. Denote by $F_n$ the free group in
$n$ letters $x_1,\ldots,x_n$. Let $\pi \colon F_n\to  L$ be the surjection
$x_i\mapsto e_i$, where $\{e_1,\ldots,e_n\}$ is a basis of $L$. Recall that for every
group $X$, we have the lower central series 
\[
X=\Gamma_1 X\supset \Gamma_2 X \supset \Gamma_3 X\supset\cdots,
\]
where $\Gamma_i X$ is defined inductively by $\Gamma_{i+1} X = [X, \Gamma_i X]$
for $i \ge 1$. In particular $\Gamma_2X$ is the commutator subgroup $[X,X]$ and $\Gamma_2\Gamma_2 X$ is
commutator subgroup of the commutator subgroup of $X$. Denote by $\Gamma_3/\Gamma_2 X$ the quotient
$\Gamma_2X/\Gamma_3X = [X,X]/[X,[X,X]]$. Notice that $\ker \pi=\Gamma_2 F_n$. The
map $\pi$ also induces a map $F_n/\Gamma_2\Gamma_2F_n \to  L$.

\begin{theorem}\label{the:compatible_action_free_group}
  Let $K$ be an arbitrary group acting on the lattice $L$. There is a compatible
  $K$-action on the Koszul resolution $P_*$ if and only if the $K$-action on $L$
  can be lifted to a $K$-action on $F_n/\Gamma_2\Gamma_2F_n$.
\end{theorem}

For the proof we need some preparation. Let $\iota\colon L\hookrightarrow \IZ L=P_0$ be
the inclusion of sets given by $l\mapsto (0)-(l)$, and define the subset
$M\subset \IZ L^n=P_1$ to be $M=\partial^{-1}(\iota(L))$, so that we get a
commutative diagram
\[ \xymatrix{ M \ar@{^{(}->}[r] \ar@{->>}[d]_{a} & \IZ L^n \ar[d]^{\partial} \\
  L \ar@{^{(}->}[r]_{\iota} & \IZ L } \] We define a new monoid structure on $M$
by $m_1\diamond m_2 = m_1+m_2-(\partial m_1)\cdot m_2$. This element of $P_1$
indeed lies in the subset $M$, because if $\partial m_i = (0)-(l_i)$ then
\begin{align}\label{thmdifferentialhomo}
  \partial(m_1+m_2- (\partial m_1) \cdot m_2 ) &= \partial(m_1)+\partial(m_2)- \bigl( (0) - (l_1) \bigr) \cdot\partial(m_2) \\
  &= (0) - (l_1) + (l_1)((0)-(l_2)) = (0) - (l_1+l_2).  \nonumber
\end{align}
The composition $\diamond$ is associative, and $0\in P_1$ serves as unit element
of $M$. Equation~\eqref{thmdifferentialhomo} shows that $a\colon M\to L$ is a
homomorphism of monoids.

\begin{remark}[Geometric picture] \label{geometric_picture}
At this point, it might be helpful to have a geometric picture in
mind. Let $X$ be the CW-complex with a $0$-cell for every element of
$L$ and a $1$-cell joining $l$ and $l+e_i$ for every $l$ and every
$i=1,2,\dots,n$. One should think of $X$ as the ``grid'' in
$\mathbb{R}^n$. Then the cellular chain complex is given by $\IZ L$ in
dimension $0$, and $\IZ L^n$ in dimension $1$, where the $\IZ$-basis
element corresponding to $[i]l$ belongs to the $1$-cell from $l$ to $l+e_i$. 
Then the differential of the cellular chain complex is
exactly the differential $\partial\colon \IZ L^n \to \IZ L$ of the
Koszul complex. The elements of $M$ can then be thought of those
$1$-chains which can be written as a sum of cycles and a single path
joining $0$ and some $l\in L$. The function $a\colon M\to L$ returns
the endpoint $l$, and the equation $m_1\diamond m_2 = m_1 + m_2\cdot
(a(m_1))$ shows that the product $\diamond$ simply translates $m_2$ in
such a way that the two paths can be concatenated, the path of
$m_1\diamond m_2$ being the concatenation of the two paths. This makes
it clear that $a$ is a homomorphism of monoids.
\end{remark}

\begin{lemma}\label{lem:M_generation}
  The monoid $M$ is a group generated by the $\IZ L$-basis elements of $P_1=\IZ
  L^n$.
\end{lemma}
\begin{proof}
  Let $a_i=[i]$ be the $\IZ L$-basis elements of $P_1$, and define the elements
  $\bar{a}_i=-(-e_i) \cdot [i]$ (with $i=1,2,\dots,n$). In our geometric picture
  $a_i$ and $\bar{a}_i$ correspond to paths from $0$ to $e_i$ and $-e_i$,
  respectively. Note that $a_i\diamond \bar{a}_i=\bar{a}_i\diamond a_i=0$, and
  define $T$ to be the submonoid of $M$ generated by all these elements. We
  first of all prove that for every $l\in L$, there are elements $\gamma_l,
  \bar{\gamma}_l\in T$ with $a(\gamma_l)=l$, $\gamma_l\diamond
  \bar{\gamma}_l=\bar{\gamma}_l\diamond \gamma_l=0$ and
  $\gamma_l+l\cdot \bar{\gamma_l}=0\in P_1$. This is true for $l=e_i$ (because
  we can then take $a_i$ and $\bar{a}_i$), and if it is true for $l,k\in L$,
  then $\gamma_{k+l}=\gamma_l\diamond \gamma_k$ and $\bar{\gamma}_{k+l} =
  \bar{\gamma}_k\diamond \bar{\gamma}_l$ do the job. Geometrically, we have
  shown so far that for every $l\in L$ there is a path $\gamma_l$ in $T$ joining
  $0$ and $l$, and the reverse path $\bar{\gamma}_l$ also belongs to $T$.

  Now let $s\in M$; we want to show that $s\in T$. By passing to $s\diamond
  \bar{\gamma}_{a(s)}$, we can assume that $a(s)=0$. Geometrically this means
  that we close the path of $s$ by joining the endpoint $a(s)$ with our chosen
  path $\gamma$; then we obtain a new element consisting of cycles only.

  But restricted to $a^{-1}(0) = \ker(\partial \colon P_1\to P_0)
  = \partial(P_2)$, $\diamond$ is just the ordinary addition in $P_1$, so it is
  enough to prove $s\in T$ for the elements
  \[ s=l\cdot \partial[i,j] = l[j] - (e_i+l)[j] - l[i] + (e_j+l)[i] \quad\text{with
    $i<j$ and $l\in L$.} \] 
  But $\gamma_l \diamond t \diamond \bar{\gamma}_l = l\cdot t\in P_1$ for 
  every $t\in M$ with $a(t)=0$, so if we prove that
  $\partial[i,j] = a_j\diamond a_i\diamond \bar{a}_j\diamond \bar{a}_i$ then we
  are done. In the space $X$, both sides correspond to the path running around
  the unit square in $e_i\times e_j$-direction, but let us give a formal
  proof. Using $a(\bar{a}_j)=-e_j$, $a(a_j)=e_j$ and $a(a_i)=e_i$ we get
  successively
  \begin{align*}
    \bar{a}_j \diamond \bar{a}_i &= \bar{a}_j + (-e_j) \cdot \bar{a}_i \\
    a_i \diamond (\bar{a}_j \diamond \bar{a}_i) &= a_i + (e_i) \bigl( \bar{a}_j + (-e_j)\cdot\bar{a}_i  \bigr)  \\
    a_j \diamond (a_i \diamond (\bar{a}_j \diamond \bar{a}_i)) &= a_j + (e_j)\biggl( a_i + (e_i) \bigl( \bar{a}_j +  (-e_j)\cdot \bar{a}_i \bigr)  \biggr)  \\
    &= a_j + (e_j)a_i + (e_i+e_j)\bar{a}_j + (e_i)\bar{a}_i \\
    &= [j] + (e_j)[i] - (e_i)[j] - [i].
  \end{align*}
  The last expression agrees with $\partial[i,j]$.
\end{proof}

\begin{lemma}\label{lem:M_bijection}
  For every element $k\in K$, the inclusions of sets $M\hookrightarrow \IZ L^n$
  and $L\hookrightarrow \IZ L$ induce a bijection of commutative diagrams
\[ 
\left\{
 \text{of groups} \quad
 \begin{array}[c]{c}  \xymatrix{ M \ar[r] \ar[d]_{a} & M  \ar[d]_{a} \\ L \ar[r]_{\rho^k} & L } \end{array}  
\right\}
 \stackrel{1:1}{\longleftrightarrow}
\left\{
 \begin{array}[c]{c} \xymatrix{ \IZ L^n \ar[r]\ar[d]_{\partial} & (\tau^k)^* \IZ L^n \ar[d]_{\partial} \\ \IZ L \ar[r]_-{\tau^k} & (\tau^k)^* \IZ L } \end{array} 
  \quad
\text{of $\IZ L$-modules} \right\}
 \]
\end{lemma}
\begin{proof}
  First of all, we show that restriction along the inclusions gives us a
  well-defined map $\Psi$ from the right to the left. Let us start with a map
  $f\colon \IZ L^n \to (\tau^k)^* \IZ L^n$ such that 
  $\partial f =  \tau^{k} \partial$. If $m\in M$, then 
  $\partial f(m) = \tau^{k} \partial m =   \tau^{k} \bigl((0) - (a(m))\bigr) = (0)-(\rho^{k} a(m))$, 
  and therefore  $f(m)\in M$, and $a(f(m)) = \rho^{k}a(m)$ so that we get a commutative square
  as desired. The restriction $f'$ of $f$ is indeed a group homomorphism:
  \begin{align*}
    f'(m_1)\diamond f'(m_2) &= f(m_1)+f(m_2) - (\partial f(m_1))\cdot f(m_2)  \\
    &= f(m_1)+f(m_2) - ( \tau^{k}(\partial m_1)) ) \cdot f(m_2) \\
    &= f(m_1 + m_2 - (\partial m_1) \cdot m_2) = f'(m_1\diamond m_2).
  \end{align*}
  $\Psi$ is injective because any two different $f$ differ at some $\IZ L$-basis
  element $[j]$ which belongs to $M$, so that the restriction $M\to M$
  still sees the difference. $\Psi$ is surjective because given $f'\colon M\to   M$, 
  we can define $f$ on basis elements by $[j]\mapsto f'([j])$ and get a
  commutative diagram as needed; then $f'$ is the restriction of $f$ because of
  Lemma~\ref{lem:M_generation}.
\end{proof}

We still have to identify the group $M$.
\begin{lemma}\label{lem:M_identification}
  The surjective map $F_n\to M$ sending the generator $x_i$ to the generator
  $a_i$ has kernel $\Gamma_2\Gamma_2 F_n$.
\end{lemma}
\begin{proof}
  Denote the kernel in question by $N$, and let $U$ be the kernel of the
  surjective map $M\xrightarrow{a} L$; this is the same as the kernel of
  $\partial \colon \IZ L^n\to  \IZ L$. We get a commutative diagram
  \[ 
  \xymatrix{ N \ar[r] \ar@{=}[d] & \Gamma_2 F_n \ar[r]\ar[d] & U\ar[d] \\ N
    \ar[r] & F_n \ar[r]\ar[d] & M \ar[d]^a \\ & L\ar@{=}[r] & L } 
  \] 
  We need to find the kernel of the map $\Gamma_2 F_n\to U$. Let $X$ be the
  CW-complex from Remark~\ref{geometric_picture}. In the cellular chain complex $\IZ L^n
  \xrightarrow{\partial} \IZ L$, $U$ agrees with the $1$-cycles, and since there
  are no $2$-cells, we get $U=H_1(X)$. Furthermore, taking $0\in L\subset X$ as
  basepoint, we have $\pi_1(X)=\Gamma_2 F_n$, and the map $\pi_1(X)=\Gamma_2
  F_n\to U=H_1(X)$ is the Hurewicz map. Therefore, $N=\Gamma_2
  \pi_1(X)=\Gamma_2\Gamma_2 F_n$.
\end{proof}

\begin{proof}[Proof of Theorem~\ref{the:compatible_action_free_group}]
  If there is a compatible action of $K$ on the Koszul complex, we in particular
  get maps $f_k\colon P_1\to  (\tau^k)^* P_1$ for every $k\in K$, satisfying
  $f_{kl}=f_kf_l$ for all $k,l\in K$. Then Lemma~\ref{lem:M_bijection} tells us
  that the $K$-action on $L$ lifts to a $K$-action on $M$. Conversely, given a
  $K$-action on $M$, the same Lemma provides us with compatible maps $f_k$, and
  we get a compatible action by~\cite[Theorem~3.1]{Adem-Ge-Pan-Petrosyan(2008)}.
\end{proof}

\begin{corollary}
  If  the map $\rho\colon L\to  L$ given by multiplication with a generator of $G \cong \IZ/m$
  can be lifted to a map $f\colon F_n\to  F_n$
  such that $f^m=1$, then there is a compatible $G$-action on $P_*$.
\end{corollary}

\begin{example}[Permutation modules]\label{exa:permutation_modules}
  Suppose that $L$ is a permutation module, so there is a $\IZ$-basis
  $e_1,\dots,e_n$ and $\rho$ acts as $e_i\mapsto e_{\sigma(i)}$ for some
  permutation $\sigma$. Then $\rho$ can be lifted by defining $f(x_i) =
  x_{\sigma(i)}$. This yields the compatible action described 
  in~\cite[Theorem~3.2]{Adem-Ge-Pan-Petrosyan(2008)}.
\end{example}

\begin{example}[Syzygies of permutation modules]
\label{exa:Syzygies_of_permutation_modules}
  Let $d$ be a divisor of $m$, and put $L = \IZ G / (1+t^d+t^{2d}+\dots +
  t^{m-d})$. Then $L$ has a $\IZ$-basis $e_0,e_1,\dots,e_{m-d-1}$, and $\rho$
  acts as follows:
  \begin{align*}
    e_i &\mapsto e_{i+1} &&\text{for $i=0,\dots,m-d-2$} \\
    e_{m-d-1} &\mapsto -e_0-e_d-e_{2d}-\dots-e_{m-2d}
  \end{align*}
  We can lift this action to $F_n$ by defining $f(x_i)=x_{i+1}$ for $i=0,1,\dots
  m-d-2$, and $f(x_{m-d-1}) = x_0^{-1}x_d^{-1}\dots x_{m-2d}^{-1}$. In order to
  check that $f^m=1$, note that
  \[ f^d(x_{m-d-1})) = f^{d-1}(f(x_{m-d-1}) = f^{d-1}(x_0^{-1}x_d^{-1}\dots
  x_{m-2d}^{-1}) = x_{d-1}^{-1}x_{2d-1}^{-1}\dots x_{m-d-1}^{-1}, \] and so
  $f^{d+1}(x_{m-d-1}) = x_d^{-1}x_{2d}^{-1}\dots x_{m-2d}^{-1} f(x_{m-d-1}^{-1})
  = x_0$, which implies $f^m(x_0)=x_0$. 

  For $d=1$ we get the augmentation ideal $L=I\subset \IZ(\IZ G)$, which was
  dealt with in~\cite[Proposition~3.3]{Adem-Ge-Pan-Petrosyan(2008)}.

\end{example}

We record the for us main important example in

\begin{lemma}\label{lem:Z[zeta]_has_compatible_action}
In the case $L = \IZ[\zeta]$ there exists a compatible group action on the Koszul resolution.
\end{lemma}
\begin{proof}
This follows from Example~\ref{exa:Syzygies_of_permutation_modules}  and~\eqref{ZG/T_and_Z[zeta]}.
\end{proof}

The approach via free groups also provides us with a technique for computing the
cohomology class $[\alpha_1]$.
\begin{lemma} \label{lem:computation_of_alpha_1}
Let $G=\IZ / m \IZ$ be a cyclic group acting on the lattice $L$.
  If the homomorphism of groups $f\colon F_n\to  F_n$ is a lift of
  $\rho^{-1}\colon L\to  L$, then the map of sets $\varphi\colon F_n\to  F_n$,
  $x\mapsto f^m(x)x^{-1}$ induces a commutative diagram of group homomorphisms
  \[ \xymatrix{ \Gamma_2/\Gamma_3 F_n \ar@{^{(}->}[r] \ar[d]_0 & F_n/\Gamma_3 F_n \ar[r]\ar[d]_{\varphi} & L \ar[d]_0 \\
    \Gamma_2/\Gamma_3 F_n \ar@{^{(}->}[r] & F_n/\Gamma_3 F_n \ar[r] & L } \] and
  the $\IZ$-dual of the connecting homomorphism of the snake lemma $L\to 
  \Gamma_2/\Gamma_3 F_n \cong \Lambda^2 L$ yields the cohomology class
  $-[\alpha_1]$.
\end{lemma}
\begin{proof}
  As a map $\varphi\colon F_n\to  F_n/\Gamma_3 F_n$ we have for $x,y\in F_n$
  \begin{align*}
    \varphi(xy) &= f^m(x)f^m(y)y^{-1}x^{-1} \\
    &= f^m(x)x^{-1} [x,f^m(y)y^{-1}] f^m(y)y^{-1} \\
    &\in f^m(x)x^{-1} f^m(y)y^{-1} \cdot \Gamma_3 F_n.
  \end{align*}
  Therefore, $\varphi$ is a group homomorphism $F_n\to  F_n/\Gamma_3
  F_n$. Furthermore, if $y$ is in $\Gamma_3F_n$ then so is $f^m(y)y^{-1}$, so
  $\varphi$ induces a homomorphism $F_n/\Gamma_3F_n\to  F_n/\Gamma_3
  F_n$.  The map $\Gamma_2/\Gamma_3 f\colon \Gamma_2/\Gamma_3 F_n=\Lambda^2
  L\to  \Gamma_2/\Gamma_3 F_n=\Lambda^2 L$ is $\Lambda^2 \rho^{-1}$, so
  that we indeed get a diagram as claimed.

  The map $f$ induces a map $M\to  M$ which in turn gives us a map
  $\tau^* \IZ L^n\to  \IZ L^n$ by Lemma~\ref{lem:M_bijection}. The latter
  can be extended to a map of chain complexes $z\colon\tau^* P_*\to  P_*$, and
  we can find a map $y\colon P_*\to  P_{*+1}$ with $y_0=0$ and $dy=z^m-1$. We
  claim that $-\id_{\IZ}\otimes_{\IZ L} y_1$ is the map $L\to \Lambda^2 L$ of
  the statement; then we are done because of
$\hom_{\IZ}(\id_{\IZ}\otimes_{\IZ L} y_1,\IZ) = \hom_{\IZ L}(y_1,\id_{\IZ})=\alpha_1$.

  Recall from the proof of Lemma~\ref{lem:M_identification} that the kernel $U$
  of $\partial\colon P_1\to  P_0$ is the kernel of the map $a\colon M\to L$,
  and we have a map $p\colon \Gamma_2 F_n\to  U$. The map 
  $P_2= \IZ L \otimes \Lambda^2 L \xrightarrow{\epsilon\otimes\id} \Lambda^2 L$ 
  factors as
  $P_2\xrightarrow{\partial} U \xrightarrow{v} \Lambda^2 L$ for some map
  $v$. When we view $U$ as the cycles of the space $X$ (as in the proof of
  Lemma~\ref{lem:M_identification}), then $v$ maps every cycle to its
  ``area''. The map $v$ is $\IZ L$-linear when we equip $\Lambda^2 L$ with the
  trivial $\IZ L$-module structure. Now we claim that the composition $\Gamma_2
  F_n\xrightarrow{p} U\xrightarrow{v} \Lambda^2L=\Gamma_2/\Gamma_3 F_n$ is the
  projection map multiplied by $(-1)$. To see this, let us start with $\gamma u
  \gamma^{-1}$ with $\gamma\in F_n$ and $u\in\Gamma_2 F_n$. Then $p(\gamma u
  \gamma^{-1})=\gamma\diamond p(u) \diamond \gamma^{-1}$, and the latter is
  easily verified to be $p(u)\cdot (a(\gamma))\in P_1$; since $v$ is $\IZ
  L$-linear, $v(p(\gamma u \gamma^{-1})) = v(p(u))$ and $v(p([\gamma,u]))=0$. We
  have therefore shown that $vp$ maps $\Gamma_3F_n$ to $0$, and now it is enough
  to verify that $[x_i,x_j]$ maps to $-e_i\wedge e_j$. But $p([x_i,x_j]) = a_i
  \diamond a_j \diamond a_i^{-1} \diamond a_j^{-1} = -\partial [i,j]$, so we
  have proved the claim.

  Finally we show that $\id_{\IZ}\otimes_{\IZ L}y_1$ is the desired map. Start with a
  generator $[i]\in P_1$; in fact, $[i]\in M$ and $x_i$ maps to $[i]$ under the
  map $\pi\colon F_n\to  M$. By construction of $z$, the element $f^m(x_i)\in
  F_n$ maps to $z_1^m([i])$ under $\pi$, and $\pi(f^m(x_i)x_i^{-1}) =
  z_1^m([i])-[i]$. Furthermore, $f^m(x_i)x_i^{-1}\in\Gamma_2 F_n$ and
  $p(f^m(x_i)x_i^{-1}) = z_1^m([i])-[i] \in U$. Now the composite
  \[
  P_1\xrightarrow{y_1} P_2 = \IZ L \otimes_{\IZ} \Lambda^2 L  \xrightarrow{\epsilon\otimes \id} \IZ \otimes_{\IZ}\Lambda^2 L  = \Lambda^2 L
  \]  applied
  to $[i]$ is the same as
  \[ v \partial y_1([i]) = v( z_1^m[i]-[i] ) = vp(f^m(x_i)x_i^{-1}), \] and we
  are done.
\end{proof}

\begin{example}
  The lemma makes it even easier to compute $\alpha_1$ in
  Lemma~\ref{lem:main_example}. The map $\rho^{-1}$ is given by the matrix
  $\left(\begin{smallmatrix} 0 & -1 & 1 \\ 1 & 0 & 0 \\ 0 & 0 &
      1 \end{smallmatrix}\right)$, so we can lift this to the free group as
  \begin{align*}
    f\colon F_3&\to  F_3 \\
    x_1 &\mapsto x_2 \\
    x_2 &\mapsto x_1^{-1} \\
    x_3 &\mapsto x_1x_3
  \end{align*}
  Then $f^4$ maps $x_1$ to $x_1$, $x_2$ to $x_2$, and
  \[ x_3 \mapsto x_1x_3 \mapsto x_2x_1x_3 \mapsto x_1^{-1} x_2 x_1x_3 \mapsto
  x_2^{-1}x_1^{-1}x_2x_1x_3 = [x_2^{-1},x_1^{-1}] x_3 \] Therefore, the map
  $L\to  \Lambda^2L$ sends $e_1,e_2$ to $0$ and $e_3$ to $-e_1\wedge
  e_2$, which is indeed $-\alpha_1$ of what we already computed in
  Lemma~\ref{lem:main_example}.  
\end{example}

%%%%%%%%%%%%%%%%%%%%%%%%%%%%%%%%%%%%%%%%%%%%%%%%%%%%%%%%%%%%%%%%%%%%%%%%%
%%%%%%%%%%%%%%% Section 5:  Proof of Theorem~ref{the:positive} %%%%%%%%%%%%%%%%%%%%%%%%%
%%%%%%%%%%%%%%%%%%%%%%%%%%%%%%%%%%%%%%%%%%%%%%%%%%%%%%%%%%%%%%%%%%%%%%%%%

\typeout{--------  Section 5:  Proof of Theorem~ref{the:positive}--------------}

\section{Proof of Theorem~\ref{the:positive}}
\label{sec:Proof_of_Theorem_ref_the:positive}

\begin{proof}[Proof of Theorem~\ref{the:positive}]
Because of Lemma~\ref{lem:reduction_to_prime}
it suffices treat the special case, where $m = p^r$ 
for some prime number $p$ and natural number $r$ and $L = \IZ(\zeta)^k = \bigoplus_{i=1}^k \IZ(\zeta)$ 
for some natural number $k$. By Lemma~\ref{lem:Z[zeta]_has_compatible_action}
there exists a compatible group action.
Now we can apply~\cite[Theorem~2.3]{Adem-Pan(2006)}). 
\end{proof}

%%%%%%%%%%%%%%%%%%%%%%%%%%%%%%%%%%%%%%%%%%%%%%%%%%%%%%%%%%%%%%%%%%%%%%%%% 
%%%%%%%%%%%%%%%%%%%%%%%%%%%%%% Section 6:  A counterexample %%%%%%%%%%%%%%%%%%%%%
%%%%%%%%%%%%%%%%%%%%%%%%%%%%%%%%%%%%%%%%%%%%%%%%%%%%%%%%%%%%%%%%%%%%%%%%%

\typeout{--------  Section 6:  A counterexample--------------}

\section{A counterexample}
\label{sec:A_counterexample}

In this section we prove the existence of a counterexample, namely,
Theorem~\ref{the:negative}.  Some preparations are needed.

The next lemma shows that the maps $\alpha_s$ can be assumed to be of a special
form.
\begin{lemma}\label{lem:mult_extension}
  Let $z_0=\tau^{-1}$ and $y_0=0$. Suppose that $z_1\colon \tau^* P_1\to P_1$
  and $y_1\colon P_1\to P_2$ are given such that $\partial z_1=z_0\partial$
  and $\partial y_1=z_1^m-1$. Then one can extend $z$ and $y$ in such a way that
  $\partial y+y\partial = z^m-1$ and the map $\alpha_s$ is the $\IZ$-dual of the composition
  \[ 
   \Lambda^s L\to L\otimes \Lambda^{s-1} L \xrightarrow{\alpha_1\dual
    \otimes 1} \Lambda^2 L\otimes \Lambda^{s-1} L\xrightarrow{\mu} \Lambda^{s+1} L. 
  \]
  Here, the first and the last map are the comultiplication and the
  multiplication of $\Lambda^* L$, respectively (see, e.g., \cite[I.2]{Akin-Buchsbaum-Weyman(1982)}).
\end{lemma}
\begin{proof}
  This follows readily from~\cite[Theorem~6 on in~II.1 page~543]{Charlap-Vasquez(1969)} 
  and its  proof, but for convenience we give an adapted version of the proof here.
  Notice that $P_i=\IZ L \otimes_{\IZ} \Lambda^i L$, and we therefore get a
  multiplication $\bullet$ on $P_*$ by tensoring the algebras $\IZ L$ and
  $\Lambda^* L$. This turns $(P_*,\partial)$ and $(\tau^* P_*,\partial)$ into graded
  commutative differential graded algebras, and there is a unique way of extending $z_1$
  multiplicatively and $\IZ L$-linearly. Explicitly, it is given by the formula
\[ l\cdot [i_1,i_2,\dots,i_i] \mapsto \rho^{-1}(l)\cdot z_1([i_1])\bullet\dots\bullet z_1([i_i]). \]
One easily checks that $z$ is a chain map. Since this map is multiplicative, we get that $z^m$ is given by
\[
z^m \colon l\cdot [i_1,i_2,\dots,i_i] \mapsto l\cdot z_1^m([i_1])\bullet\dots\bullet z_1^m([i_i]). 
\]

Now we define $y$ to be the $\IZ L$-linear map given by
\begin{align}\label{ydefinition}
  y\colon [i_1,i_2,\dots,i_i] \mapsto \sum_{j=1}^i (-1)^{j-1}
  z_1^m([i_1])\bullet\dots\bullet z_1^m([i_{j-1}])\bullet y_1([i_j])\bullet
  [i_{j+1}]\bullet\dots\bullet [i_i].
\end{align}
Next we show that $\partial y+y\partial = z^m-1$ as maps $P_s\to P_s$ for $s = 0,1,2 \ldots$.
We proceed by induction over $s$. The induction beginning $s = 0,1$ follows from the definitions and assumptions,
the induction step from $s-1$ to $s\ge 2$ is done as follows. 
Consider $\alpha = [i_1,\ldots ,i_a]$  with $ a \ge 1$ and $i_1  < \cdots < i_a$ and $\beta = [j_1,\ldots,j_b]$ 
with $b \ge 1$ and $j_1 < \cdots<  j_b$ such that $a +b = s$ and $i_a < j_1$.
Then
\begin{align}\label{yrelation}
  y(\alpha\bullet \beta)=y(\alpha)\bullet \beta +(-1)^{|\alpha|} z^m(\alpha)\bullet y(\beta)
\end{align}
by definition. This observation generalizes as follows: let us call a pair $\alpha,\beta\in P_s$ admissible
if we can write $\alpha=\sum_r x_r[i_1^r,\dots,i_a^r]$ and $\beta=\sum_s y_s[j_1^s,\dots,j_b^s]$ with
$i_1^r<\dots<i_a^r<j_1^s<\dots<j_b^s$ and $x_r,y_s\in \IZ L$ for all $r,s$. Then \eqref{yrelation} holds, and every
element in $P_s$ is a linear combination of elements of the form $\alpha\bullet\beta$ with $(\alpha,\beta)$ 
admissible, so it is enough to prove $(\partial y+y\partial)(\alpha\bullet\beta)=(z^m-1)(\alpha\bullet\beta)$ in
that case. 

We directly deduce from \eqref{yrelation} that
\[ \partial y(\alpha\bullet \beta) = \partial y(\alpha) \bullet \beta - (-1)^{|\alpha|}
y(\alpha)\bullet \partial \beta + (-1)^{|\alpha|} \partial z^m(\alpha) \bullet y(\beta) +
z^m(\alpha)\bullet \partial y(\beta). \] 
On the other hand, 
$\partial(\alpha\bullet \beta) = \partial \alpha\bullet \beta + (-1)^{|\alpha|} \alpha\bullet \partial \beta,$ 
and since the pairs $(\partial\alpha,\beta)$ and $(\alpha,\partial\beta)$ are admissible as well,
 we can use~\eqref{yrelation} again and get
\[ 
y(\partial(\alpha\bullet \beta)) = y(\partial \alpha)\bullet \beta - (-1)^{|\alpha|} z^m(\partial
\alpha)\bullet y(\beta) + (-1)^{|\alpha|} y(\alpha)\bullet \partial \beta + z^m(\alpha)\bullet y(\partial \beta).
\] 
Adding these equations using the induction hypothesis $\partial y+y\partial (\alpha) =(z^m-1)(\alpha)$
and $\partial y+y\partial (\beta) =(z^m-1)(\beta)$, we get
\[ 
(\partial y+y\partial)(\alpha\bullet \beta) = (z^m-1)(\alpha)\bullet \beta + z^m(\alpha)\bullet (z^m-1)(\beta) = (z^m-1)(\alpha\bullet \beta). 
\] 
Having defined $y$ and $z$ and shown  $\partial y+y\partial = z^m-1$, we are now
ready to finish the proof of Lemma~\ref{lem:mult_extension}. Notice that there is a natural isomorphism 
$\hom_{\IZ  L}(X,\IZ)\cong \hom_{\IZ}(\IZ \otimes_{\IZ L} X,\IZ)$ for $\IZ L$-modules $X$,
so it remains to compute $\bar{y}=\id_{\IZ} \otimes_{\IZ L} y$. Using the fact that
$\id_{\IZ} \otimes_{\IZ L} z^m$ is the identity, we get from~\eqref{ydefinition} that
$\bar{y}_i\colon \Lambda^i L \to \Lambda^{i+1} L$ is given by
\[
  e_{i_1}\wedge\dots\wedge e_{i_i} \mapsto \sum_{j=1}^i (-1)^{j-1}
  e_{i_1}\wedge\dots\wedge \bar{y}_1(e_{i_j})\wedge\dots\wedge e_{i_i}.
\]
This agrees with the composition given in the statement of the lemma.
\end{proof}

\begin{theorem}\label{the:direct_sum_theorem}
  Suppose that $L=X\oplus \Lambda^2 X\dual$ for some $\IZ G$-module $X$ whose
  underlying $\IZ$-module is free of finite rank. If in the Lyndon-Hochschild-Serre spectral
  sequence the $d_2$-differential
  \[ H^*(G,H^3(L))\to H^*(G,H^2(L)) \] is zero, then the class 
  $[\alpha_1^X]\in \Ext^2_{\IZ G}(\Lambda^2 X\dual,X\dual)$ vanishes.
\end{theorem}
\begin{proof}
  Let $L=X\oplus Y$ for $\IZ G$-modules $X$ and $Y$. We know that $X\otimes Y$
  is a direct summand of $\Lambda^2 L$ as $\IZ G$-module from the exponential law~\eqref{exponential_law}, and similarly
  $\Lambda^2 X\otimes Y$ is a direct summand of $\Lambda^3 L$. In the the sequel we denote by $\iota$ 
   the inclusions of and by $\pi$ obvious projections onto direct summands. 

  Next we prove that  the  diagram
  \begin{align}\label{alphadiagram}
    \xymatrix@!C=7em{\Lambda^2 L \ar[r]^{\alpha_2\dual} & \Lambda^3 L \ar[d]^{\pi} \\
      X\otimes Y \ar[u]^{\iota} \ar[r]_{(\alpha_1^X)\dual\otimes \id} & \Lambda^2
      X \otimes Y }
  \end{align}
  commutes if we choose the maps $z$, $y$ carefully. To do so, let $P^X_*$,
  $P^Y_*$ be the Koszul complexes associated with $X$ and $Y$, respectively, and
  choose maps 
  $z_1^X \colon \IZ X \otimes_{\IZ} X \to \IZ X \otimes_{\IZ} X$, 
  $y_1^X \colon \IZ X \otimes_{\IZ} X \to \IZ X \otimes_{\IZ} \Lambda^2 X$, 
  $z_1^Y \colon \IZ Y \otimes_{\IZ} Y \to \IZ Y \otimes_{\IZ} Y$, 
  and $y_1^Y \colon \IZ Y \otimes_{\IZ} Y \to \IZ Y \otimes_{\IZ} \Lambda^2 Y$.
  Define
  \begin{align*}
    z_1^L &= \bigl(\id_{\IZ L} \otimes_{\IZ X} z_1^X \bigr) \oplus \bigl( \id_{\IZ L} \otimes_{\IZ Y} z_1^Y \bigr) \colon  \tau^*(\IZ L\otimes L)
   \to \IZ L\otimes_{\IZ} L; 
    \\
    y_1^L &
    = \bigl(\id_{\IZ L} \otimes_{\IZ} \iota \bigr)\circ \bigl(\bigl(\id_{\IZ L} \otimes_{\IZ X} y_1^X \bigr) \oplus \bigl(\id_{\IZ L} \otimes_{\IZ Y} y_1^Y\bigr)\bigr) 
   \colon \IZ L\otimes_{\IZ} L \to \IZ L \otimes_{\IZ} \Lambda^2 L.
  \end{align*}
   By definition, the diagram
  \[ 
   \xymatrix@!C=5em{ L \ar[r]^{\id_{\IZ}\otimes_{\IZ L} y_1^L} & \Lambda^2 L \\ X
    \ar[u]^{\iota} \ar[r]_{\id_{\IZ} \otimes_{\IZ X} y_1^X} & \Lambda^2 X
    \ar[u]^{\iota} } 
   \] 
  commutes (and similarly for $Y$). Now we use 
  Lemma~\ref{lem:mult_extension} to get maps $z^L$ and $y^L$,.

 It remains to  show that $X\otimes Y\xrightarrow{\bar{y}^X\otimes \id} \Lambda^2 X \otimes Y$
  equals the composition
  \[ X\otimes Y \xrightarrow{\iota} \Lambda^2 L \xrightarrow{\nabla} L\otimes L
  \xrightarrow{\bar{y}^L\otimes \id} \Lambda^2 L\otimes L \xrightarrow{\mu}
  \Lambda^3 L \xrightarrow{\pi} \Lambda^2 X\otimes Y 
  \] 
  where $\bar{y}^L=\id_{\IZ} \otimes_{\IZ L} y$, $\nabla$ 
  is the comultiplication and $\mu$
  the multiplication of $\Lambda^* L$. So let us start with $a\otimes b\in X\otimes Y$; then
  $\bar{y}^L\otimes 1 \circ \nabla$ maps it to $\bar{y}(a)\otimes
  b-\bar{y}(b)\otimes a\in \Lambda^2 X\otimes Y\oplus \Lambda^2 Y\otimes
  X\subset \Lambda^2 L\otimes L$. But $\pi\mu$ is the identity on the first
  summand and zero on the second one. We have therefore shown 
  that~\eqref{alphadiagram} commutes. 

  Dualizing the diagram~\eqref{alphadiagram}  yields
  \[ \xymatrix@!C=10em{
    \Lambda^2 L\dual \ar[d] 
    & 
   \Lambda^3 L\dual \ar[l]_{\alpha_2^L} 
   \\
    X\dual \otimes Y\dual \ar[d]_-{\cong}
    & 
   \Lambda^2 X\dual \otimes Y\dual \ar[u] \ar[l]^{\alpha_1^X \otimes \id}    
    \\
    \hom_{\IZ} (Y,X\dual)  
    & 
    \hom_{\IZ} (Y,\Lambda^2 X\dual) \ar[u]_-{\cong}      \ar[l]^{\hom_{\IZ}(\id_Y,\alpha_1^X)} } 
   \] 
  Now put  $Y=\Lambda^2 X\dual$. Then the bottom row maps $\id_{\Lambda^2 X\dual}$ to
  $\alpha_1^X\in\hom_{\IZ}(\Lambda^2 X\dual,X\dual)$. This implies that the map
  \begin{multline*}
   \Ext^2_{\IZ G} \bigl(\id_{\IZ}, \hom_{\IZ}(\id_{\Lambda^2 X\dual},\alpha_1^X)\bigr) \colon 
  \Ext^2_{\IZ G} \bigl(\IZ, \hom_{\IZ}(\Lambda^2 X\dual,\Lambda^2 X\dual)\bigr) 
   \\
  \to 
  \Ext^2_{\IZ G}  \bigl(\IZ, \hom_{\IZ}(\Lambda^2 X\dual,X\dual)\bigr) \cong \Ext^2_{\IZ G}  \bigl(\Lambda^2 X\dual,X\dual)\bigr)
\end{multline*}
contains the class $[\alpha_1^X]$ in its image. 
The second differential $d_2$ is by assumption zero and agrees by Lemma~\ref{lem:differential_is_product_with_alpha} with
the composite
\[
\Ext^2_{\IZ G} (\id_{\IZ}, \alpha_2^L) \colon \Ext^2_{\IZ G}  (\IZ, \Lambda^3L\dual) \to \Ext^2_{\IZ G}  (\IZ, \Lambda^2L\dual) 
\xrightarrow{\cong} \Ext^4_{\IZ G}  (\IZ, \Lambda^2L\dual), 
\]
where the last map is the periodicity isomorphism. We conclude from the diagram above  that  the map
$ \Ext^2_{\IZ G} \bigl(\id_{\IZ}, \hom_{\IZ}(\id_{\Lambda^2 X\dual},\alpha_1^X)\bigr) $ is trivial and hence $[\alpha_1^X]$ vanishes.
\end{proof}

\begin{proof}[Proof of Theorem~\ref{the:negative}] 
This follows directly from 
Theorem~\ref{the:direct_sum_theorem} and Lemma~\ref{lem:main_example}.
\end{proof}
  
In order  prove Corollary~\ref{cor:m_divisible_by_four}, we need:

\begin{lemma}\label{lem:general_cyclic}
       Let $G'\to G$ be a surjection of finite cyclic groups, and let us regard any $\IZ G$-module as $\IZ G'$-module via this
       map.
       \begin{enumerate}
         \item \label{lem:general_cyclic:injective} 
       For every $\IZ G$-module $X$ whose underlying $\IZ$-module is free, the induced map 
       $H^2(G,X)\to H^2(G',X)$ is injective;
         \item \label{lem:general_cyclic:class} 
       Let $L$ be a $\IZ G$-module as above; then the class $[\alpha_1^G]$ maps to the class $[\alpha_1^{G'}]$ 
       under the map $H^2(G,\hom_{\IZ}(\Lambda^2 L,L))\to H^2(G',\hom_{\IZ}(\Lambda^2 L,L))$.
     \end{enumerate}
   \end{lemma}
    \begin{proof}~\ref{lem:general_cyclic:injective} 
    The spectral sequence for the extension $\IZ / d\IZ \to \IZ / dm\IZ \to \IZ / m\IZ$ yields an exact sequence
        \[ 
        \cdots \to H^0(\IZ/m,H^1(\IZ / d\IZ, X)) \to H^2(\IZ / m\IZ, X) \to H^2(\IZ / dm\IZ, X) \to \cdots.
        \]
        Since $X$ is torsion-free and $\IZ/d$ acts trivially on it, $H^1(\IZ / d\IZ, X)$ and hence also the first group are trivial, and therefore the second map is injective.
        \\[1mm]~\ref{lem:general_cyclic:class} This follows from~\cite[I.2, Theorem~3 in I.2 on page~538]{Charlap-Vasquez(1969)}.
      \end{proof}
 \smallskip
\begin{proof}[Proof of Corollary~\ref{cor:m_divisible_by_four}]~\ref{cor:m_divisible_by_four:non-trivial}
This follows from Theorem~\ref{the:negative}  and Lemma~\ref{lem:general_cyclic}.
\\[1mm]~\ref{cor:m_divisible_by_four:trivial}
This follows from~\cite[Corollary in~II.1 on page~543]{Charlap-Vasquez(1969)}, Lemma~\ref{lem:differential_is_product_with_alpha}
and Lemma~\ref{lem:alpha_and_v_agree}.
\end{proof}

%%%%%%%%%%%%%%%%%%%%%%%%%%%%%%%%%%%%%%%%%%%%%%%%%%%%%%%%%%%%%%%%%%%%%%%%%
%%%%%%%%%%%%%%% Section 7:  Group cohomology and the equivariant Eulercharacteristic %%%%%%%%%%%
%%%%%%%%%%%%%%%%%%%%%%%%%%%%%%%%%%%%%%%%%%%%%%%%%%%%%%%%%%%%%%%%%%%%%%%%%

\typeout{--------  Section 7:  Group cohomology and the equivariant Euler characteristic--------------}

\section{Group cohomology and the equivariant Euler characteristic}
\label{sec:Group_cohomology_and_the_equivariant_Eulercharacteristic}

In this section we relate the cohomology of $\Gamma$ to the equivariant Euler characteristic of 
the finite $G$-$CW$-complex $L\backslash \underline{E}\Gamma$.

Let $\Sw(G)$ be \emph{Swan's group}, i.e., generators
are isomorphism classes $[M]$ of $\IZ G$-modules $M$ which are finitely
generated as abelian groups, and every short exact sequence 
$0 \to M_0 \to M_1 \to M_2 \to 0$ of such modules yields the relation 
$[M_0]- [M_1] + [M_2] = 0$. Next we define a homomorphism
\begin{eqnarray}
\widehat{h} \colon \Sw(G) \to \IQ^{>0}
\label{h:Sw(G)_to_Q_ge}
\end{eqnarray}
to the multiplicative group of positive rational numbers.  It sends the class of
a $\IZ G$-module $M$ which is finitely generated as abelian group to
$\frac{|\widehat{H}^0(G;M)|}{|\widehat{H}^1(G;M)|}$. Notice that
$\widehat{H}^i(G;M)$ is a finite group for such $M$.  In order to show that this
is well-defined, we have to check for an exact sequence 
$0 \to M_0 \to M_1 \to M_2 \to 0$ of $\IZ G$-modules which are finitely generated as abelian groups
\[\frac{|\widehat{H}^0(G;M_1)|}{|\widehat{H}^1(G;M_1)|} 
= 
\frac{|\widehat{H}^0(G;M_0)|}{|\widehat{H}^1(G;M_0)|} \cdot \frac{|\widehat{H}^0(G;M_2)|}{|\widehat{H}^1(G;M_2)|}.\]
This follows from the induced long exact sequence (see~\cite[(5.1) in~VI.5 on page~136]{Brown(1982)})
\begin{multline*}
\cdots \to \widehat{H}^i(G;M_0)  \to \widehat{H}^i(G;M_1)  \to \widehat{H}^i(G;M_2)
 \\ 
\to \widehat{H}^{i+1}(G;M_0)  
\to \widehat{H}^{i+1}(G;M_1)  \to \widehat{H}^{i+1}(G;M_2) \to \cdots
\end{multline*}
which is compatible with the cup-product 
(see~\cite[(5.6) in~VI.5 on page~136]{Brown(1982)}), and from the $2$-periodicity of the Tate cohomology
$\widehat{H}^i(G;M) \xrightarrow{\cong} \widehat{H}^{i+2}(G;M)$ coming from the
cup-product with a generator of $\widehat{H}^0(G;\IZ) \cong \IZ/m$
(see~\cite[Theorem~9.1 in~VI.9 on page~154]{Brown(1982)}).

Given a $\IZ G$-module $M$ which is finitely generated as abelian group, 
define \emph{its homological Euler characteristic} by
\begin{eqnarray}
\chi^G_h(M) := \sum_{i \ge 0} (-1)^i \cdot [H^i(M)] \in \Sw(G).
\label{homological-Euler_characteristic}
\end{eqnarray}

\begin{lemma} \label{lem:widehat_and_homological_Euler_characteristic}
 We get for all integers $k$ with $2k > n$
\[
\widehat{h}\bigl(\chi^G_h(L)\bigr) = \frac{|H^{2k}(\Gamma)|}{|H^{2k+1}(\Gamma)|}.
\]
\end{lemma}
\begin{proof}
Let $E^{*,*}_r$ be the $E_r$-term in the Lyndon-Hochschild-Serre spectral sequence associated to
$\Gamma = L \rtimes_{\phi} G$. Notice for the sequel that $E^{i,j}_2 = H^i(G;H^j(L))$ vanishes for
$j > n$ and is finite for $i > 0$ and hence the same statement holds for $E_r^{i,j}$ for $r = 3,4 \ldots$ and $r = \infty$.

We first show for $r \ge 2$ the following equality
\begin{eqnarray}
\prod_{i,j, i+j \in \{2k,2k+1\}} \; |E^{i,j}_r|^{(-1)^{i+j}} 
& = & 
\prod_{i,j, i+j \in \{2k,2k+1\}} \; |E^{i,j}_{r+1}|^{(-1)^{i+j}}. 
\label{passage_from_Er_to_Er_plus_1}
\end{eqnarray}
Notice that $E^{i,j}_r$ vanishes if $i$ or $j$ is negative.
The differentials in the $E_r$-term yield for non-negative integers $a$ and $b$
$\IZ G$-chain complexes $C_*^{(a,b)}$ of $\IZ G$-modules which are finitely
generated as abelian groups if we put
\[
C_* = E^{a + r\cdot \ast,b- (r-1) \cdot \ast}_r
\]
If $a + b > n$, then $C^{(a,b)}_*$ is a finite-dimensional chain complex of finite abelian groups
and hence we get
\begin{eqnarray*}
\prod_{l \in \IZ} \; |C_l^{(a,b)}|^{(-1)^l}
& = & 
\prod_{l \in \IZ} \; \bigl|H_l(C_*^{(a,b)})\bigr|^{(-1)^l}.
\end{eqnarray*}
Since $H_l(C_*^{(a,b)}) = E_{r+1}^{a + r\cdot l,b- (r-1) \cdot l}$, we conclude provided that $a  + b > n$ holds
\begin{eqnarray*}
\prod_{l \in \IZ} \;\bigl|E^{a + r\cdot l ,b- (r-1)\cdot l}_r\bigr|^{(-1)^l}
& = &
\prod_{l \in \IZ} \; \bigl|E_{r+1}^{a + r\cdot l,b- (r-1)\cdot l}\bigr|^{(-1)^l}.
\end{eqnarray*}
If we let $a$ run through $\{0,1,\ldots, (r-1)\}$ and $b$ through $\{2k+ j \mid j = 0,1\}$
and take the product of the equalities above raised to the $(-1)^{a+j}$-th power for these values, we conclude
\begin{multline*}
\prod_{a = 0}^{r-1} \;\prod_{j = 0}^{1} \;\prod_{l \in \IZ} \;\bigl|E^{a + r\cdot l , 2k+j- (r-1)\cdot l}_r\bigr|^{(-1)^{a + j + l}}
\\
=  
\prod_{a = 0}^{r-1} \;\prod_{j = 0}^{1} \;\prod_{l \in \IZ} \; \bigl|E^{a + r\cdot l , 2k+j- (r-1)\cdot l}_{r+1}\bigr|^{(-1)^{a + j + l}}.
\end{multline*}
One easily checks
\begin{eqnarray*}
a +j + l \equiv 0 \mod 2 & \Leftrightarrow & 2k - (a + r\cdot l) \equiv 2k + j - (r-1) \cdot l\mod 2
\\
a +j + l \equiv 1 \mod 2 & \Leftrightarrow & 2k + 1 - (a + r\cdot l) \equiv 2k + j - (r-1) \cdot l \mod 2
\end{eqnarray*}
In the Lyndon-Hochschild-Serre spectral sequence the cup product with a
generator $\mu \in E^{2,0}_2 = H^2(G;H^0(L)) = H^2(G) \cong \IZ/m$ induces
isomorphisms 
$E^{i,j}_2 = H^i(G;H^j(L)) \xrightarrow{\cong} E^{i+2,j}_2 = H^{i+2}(G;H^j(L))$ 
for $i > 0$ and $j \ge 0$.  All differentials starting or
ending at $E^{2,0}_r$ are zero since the edge homomorphism $H^2(G) \to
H^2(\Gamma)$ is injective. Hence $E^{2,0}_2 = E^{2,0}_r = E^{2,0}_{\infty}$ and
the cup product with $\mu$ induces isomorphisms 
$E^{i,j}_r \xrightarrow{\cong} E^{i+2,j}_r $ for $i + j > n$ 
since an isomorphism of chain complexes induces an
isomorphism on homology.  This implies
\[
\begin{array}{lcll}
\bigl|E^{a + r\cdot l , 2k+j- (r-1)\cdot l}_r\bigr|^{(-1)^{a + j + l}} 
& = & \bigl|E^{a + r\cdot l , 2k - (a + r\cdot l) }_r\bigr|^{(-1)^{2k}} 
& \text{if} \;
a + j + l  \equiv 0 \mod 2;
\\
\bigl|E^{a + r\cdot l , 2k+j- (r-1)\cdot l}_r\bigr|^{(-1)^{a + j + l}} 
& = & \bigl|E^{a + r\cdot l , 2k +1 - (a + r\cdot l)}_r\bigr|^{(-1)^{2k+1}} 
& \text{if} \;
a + j + l  \equiv 1 \mod 2,
\end{array}
\]
and analogously for $r$ replaced by $(r+1)$. Hence we obtain
\[
\prod_{a = 0}^{r-1} \;\prod_{c = 0}^{1} \;\prod_{l \in \IZ} \;\bigl|E^{a + r\cdot l , 2k+c - (a + r \cdot l)}_r\bigr|^{(-1)^c}
=  
\prod_{a = 0}^{r-1} \;\prod_{c = 0}^{1} \;\prod_{l \in \IZ} \;\bigl|E^{a + r\cdot l , 2k+c - (a + r \cdot l)}_{r+1}\bigr|^{(-1)^c}.
\]
But this is the same as the desired equality~\eqref{passage_from_Er_to_Er_plus_1}
since there is the bijection
\[
\{(a,c,l) \mid a\in\{0,1,\dots,r-1\}, c\in\{0,1\}, l\in\IZ\} \to\{(i,j) \mid i,j\in\IZ, i+j\in\{2k,2k+1\} \}
\]
given by $(a,c,l)\mapsto (a+r\cdot l, 2k+c-(a+r\cdot l))$. 
This finishes the proof of~\eqref{passage_from_Er_to_Er_plus_1}. 

We conclude from~\eqref{passage_from_Er_to_Er_plus_1} by induction over $r \ge 2$.
\begin{eqnarray*}
\prod_{i,j, i+j \in \{2k,2k+1\}} |E^{i,j}_{2}|^{(-1)^{i+j}} 
& = & 
\prod_{i,j, i+j \in \{2k,2k+1\}} |E^{i,j}_{\infty}|^{(-1)^{i+j}}. 
\end{eqnarray*}
Since the terms $E_{\infty}^{i,j}$  are quotients of a filtration of $H^{i+j}(\Gamma)$ and
$E_2^{i,j} = H^i(G;H^j(L))$, this implies
\begin{eqnarray*}
\prod_{i,j, i+j \in \{2k,2k+1\}} |H^i(G;H^j(L))|^{(-1)^{i+j}} 
& = & 
\frac{|H^{2k}(\Gamma)|}{|H^{2k+1}(\Gamma)|}.
\end{eqnarray*}
Since $\widehat{H}^i(G,H^j(L)) \cong H^i(G;H^j(L)) \cong \widehat{H}^{i+2}(G,H^j(L))  \cong H^{i+2}(G;H^j(L))$ holds for $i > 0$,
we conclude
\begin{eqnarray*}
\widehat{h}\bigl(\chi^G_h(L)\bigr)
& = & 
\widehat{h}\biggl(\sum_{j \ge 0} (-1)^j \cdot H^j(L)\biggr)
\\
& = & 
\prod_{i = 0}^1 \prod_{j \ge 0} \; \bigl|\widehat{H}^i(G;H^j(L))\bigr|^{(-1)^{i+j}}
\\
& = & 
\prod_{i,j, i+j \in \{2k,2k+1\}} |H^i(G;H^j(L))|^{(-1)^{i+j}} 
\\
& = & 
\frac{|H^{2k}(\Gamma)|}{|H^{2k+1}(\Gamma)|}.
\end{eqnarray*}
\end{proof}

Let $A(G)$ be the \emph{Burnside ring of $G$}, i.e., the Grothendieck construction applied to the
semi-ring of isomorphisms classes of finite $G$-sets under disjoint union and cartesian product.
Given a finite $G$-$CW$-complex $X$, define its \emph{$G$-Euler characteristic}
\begin{eqnarray}
\chi^G(X) \in A(G)
\label{G_Euler-characteristic}
\end{eqnarray}
by the sum $\sum_{c} (-1)^{\dim(c)} \cdot [t(c)]$, where $c$ runs though the equivariant cells of $X$,
$\dim(c)$ is the dimension of $c$ and $t(c)$ is given by the orbit though one point in the interior of $c$. 
If $c$ is obtained by attaching $G/H \times D^k$, then $\dim(c) = k$ and $t(c) = G/H$.
Let
\begin{eqnarray}
r \colon A(G) & \to & \Sw(G)
\label{map_A(G)_to_sw(G)}
\end{eqnarray}
be the map sending the class of a finite $G$-set $S$ to the associated $\IZ G$-permutation module $\IZ[S]$ with $S$ as $\IZ$-basis.
\begin{lemma} \label{lem:chiG_and_chG_h}
Let $X$ be a finite $G$-$CW$-complex. Then
\[
r(\chi^G(X)) = \chi^G_h(X).
\]
\end{lemma}
\begin{proof}
This follows from the following computation in $\Sw(G)$ based on the fact
that $C_k(X) \cong_{\IZ G} \bigoplus_{c, \dim(c) = k} \IZ[t(c)]$
\begin{eqnarray*}
r(\chi^G(X)) 
& = & 
r \biggl(\sum_c (-1)^{\dim(c)} \cdot [t(c)]\biggr)
\\
& = &
\sum_c (-1)^{\dim(c)} \cdot \IZ[t(c)]
\\
& = & 
\sum_{k \ge 0} (-1)^k \cdot [C_k(X)]
\\
& = & 
\sum_{k \ge 0} (-1)^k \cdot [H_k(X)]
\\
& = &
\chi^G_h(X).
\end{eqnarray*}
\end{proof}

\begin{theorem}[Group cohomology and the equivariant Eulercharacteristic]
  \label{the:Group_cohomology_and_equivariant_Eulercharacteristic}
  Let $k$ be an integer such that $2k > n$. Then $H^{2k}(\Gamma)$ and
  $H^{2k+1}(\Gamma)$ are finite and
  \[
  \frac{|H^{2k}(\Gamma)|}{|H^{2k+1}(\Gamma)|} =
  \widehat{h} \circ r\bigl(\chi^G(L\backslash \underline{E}\Gamma)\bigr)
  \]
  where the homomorphism of abelian groups given by the composite $\widehat{h} \circ r \colon A(G) \to \IQ^{> 0}$
  sends $[G/H]$ to $|H|$.
\end{theorem}
\begin{proof} We compute
\[
\widehat{h} \circ r\bigl([G/H]\bigr)
=
\frac{\widehat{H}^0(G;\IZ[G/H])}{\widehat{H}^1(G;\IZ[G/H])}
=
\frac{\widehat{H}^0(H;\IZ)}{\widehat{H}^1(H;\IZ)}
= \frac{|H|}{1} = |H|.
\]
Now apply Lemma~\ref{lem:widehat_and_homological_Euler_characteristic} and
Lemma~\ref{lem:chiG_and_chG_h}.
\end{proof}

\begin{remark}[$L^2$-Euler  characteristic]\label{L2-Euler_characteristic}
In~\cite[Definition~6.83 and Definition~6.84 on page~281]{Lueck(2002)} a Burnside group $A(\Gamma)$ and a 
$\Gamma$-$CW$-Euler characteristic 
\begin{eqnarray}
& \chi^{\Gamma}(\underline{E}\Gamma) \in A(\Gamma) &
\label{chiGamma}
\end{eqnarray}
is defined. There is a homomorphism  of groups
\begin{eqnarray}
q \colon A(\Gamma) & \to & A(G)
\label{A(Gamma)_to_A(G)}
\end{eqnarray}
which sends the class $[S]$ of a proper cocompact $\Gamma$-set $S$ 
to the class $[L\backslash S]$ of the finite $G$-set $L \backslash S$.
One easily checks
\begin{eqnarray*}
q\bigl(\chi^{\Gamma}(\underline{E}\Gamma)\bigr) & = & \chi^G(L\backslash \underline{E}\Gamma).
\label{chiGamma_and_chiG}
\end{eqnarray*}
There is an injective homomorphism called \emph{global} $L^2$-character map
(see~\cite[Definition~6.86 on page~282]{Lueck(2002)})
\begin{eqnarray}
\ch^{\Gamma} \colon A(\Gamma) \to \prod_{(K)} \IQ
\label{global_character_map}
\end{eqnarray}
where $(K)$ runs through the conjugacy classes of finite subgroups of $\Gamma$.
It is rationally an isomorphism. Since $\Gamma$ is amenable, we conclude 
from~\cite[Lemma~6.93 on page~284]{Lueck(2002)})
\begin{eqnarray}
\ch^{\Gamma}\bigl(\chi^{\Gamma}(\underline{E}\Gamma)\bigr)_{(K)}
=
\begin{cases}
0 & \text{if}\; |W_{\Gamma}K| = \infty;
\\
\frac{1}{|W_{\Gamma}K|}  & \text{if}\; |W_{\Gamma}K| < \infty.
\end{cases}
\label{chGamma_(chiGamma_(underlineEGamma)}
\end{eqnarray}
where $W_{\Gamma}K := N_{\Gamma}K/K$.
\end{remark}

\begin{example}[$G$-action has non-trivial fixed point]
\label{exa:G-action_has_non-trivial_fixed_point}
Suppose that $L^G \not = 0$. Then $|W_{\Gamma} K| = 0$ for all finite subgroups $K \subseteq \Gamma$,
and we conclude from Theorem~\ref{the:Group_cohomology_and_equivariant_Eulercharacteristic}
and Remark~\ref{L2-Euler_characteristic} for $2k > n$
\[
|H^{2k}(\Gamma) = |H^{2k+1}(\Gamma)|.
\]
\end{example}

\begin{example}\label{exa:n_is_p_for_group_cohomology_and_equivaraint_Euler}
  Suppose that $m = p$ for a prime $p$. 
  If $G$ acts not free outside the origin on $L$, we conclude from
  Example~\ref{exa:G-action_has_non-trivial_fixed_point}.
  \[|H^{2k}(\Gamma)| = H^{2k+1}(\Gamma)|.\] Suppose that $G$ acts free outside
  the origin on $L$.  Let $\calp$ be a complete set of representatives of the
  conjugacy classes $(P)$ of finite non-trivial subgroups $P \subseteq
  \Gamma$.  Notice that for each $P$ the projection $\Gamma \to G$ induces an
  isomorphism $P \xrightarrow{\cong} G \cong \IZ/p$ and we have $W_{\Gamma} P =  \{1\}$. 
We conclude from Remark~\ref{L2-Euler_characteristic} by inspecting
  the definition of the global character map $\ch^{\Gamma}$ (see~\cite[Example~6.94 on page~184]{Lueck(2002)})
  \[
  \chi^{\Gamma}(\underline{E}\Gamma) = \frac{-|\calp|}{p} \cdot [\Gamma] +
  \sum_{P \in \calp} \;[\Gamma/P] \quad \in A(\Gamma).
  \]
  and hence
  \begin{eqnarray*}
    \chi^G(L \backslash \underline{E}G) & = & - \frac{|\calp|}{p}   \cdot [G] +   |\calp| \cdot [G/G] \quad \in A(G).
  \end{eqnarray*}
  We mention that
  \[|\calp| = |H^1(G;L)| = \bigl |(L\backslash \underline{E}\Gamma)^G\bigr| =  p^s\] 
   by~\cite[Lemma~1.9]{Davis-Lueck(2010)}, where $s$ is the integer
  uniquely determined by $N = (p-1)\cdot s$.  Theorem~\ref{the:Group_cohomology_and_equivariant_Eulercharacteristic}
  implies for $2k > n$
  \begin{eqnarray*}
   \frac{|H^{2k}(\Gamma)|}{|H^{2k+1}(\Gamma)|} & = &   p^s.
  \end{eqnarray*}
  All this is consistent with the
  computation in Theorem~\ref{the:cohomology_of_Gamma} for $2k > n$
  \begin{eqnarray*}
    H^{2k}(\Gamma) & = & \prod_{(P) \in \calp} H^{2k}(M)  \cong \prod_{P \in \calp} \IZ/p
    \\
    H^{2k+1}(\Gamma) & = & 0.
  \end{eqnarray*}
\end{example}

%%%%%%%%%%%%%%%%%%%%%%%%%%%  References %%%%%%%%%%%%%%%%%%%%%%%%%%%%%%%%%%%%%%%%%%%%%

\typeout{-------------------------------------  References    -------------------------------}

%\bibliographystyle{abbrv}

%\bibliography{dbpub,dbpre}

\begin{thebibliography}{10}

\bibitem{Adem-Ge-Pan-Petrosyan(2008)}
A.~Adem, J.~Ge, J.~Pan, and N.~Petrosyan.
\newblock Compatible actions and cohomology of crystallographic groups.
\newblock {\em J. Algebra}, 320(1):341--353, 2008.

\bibitem{Adem-Pan(2006)}
A.~Adem and J.~Pan.
\newblock Toroidal orbifolds, {G}erbes and group cohomology.
\newblock {\em Trans. Amer. Math. Soc.}, 358(9):3969--3983 (electronic), 2006.

\bibitem{Akin-Buchsbaum-Weyman(1982)}
K.~Akin, D.~A. Buchsbaum, and J.~Weyman.
\newblock Schur functors and {S}chur complexes.
\newblock {\em Adv. in Math.}, 44(3):207--278, 1982.

\bibitem{Atiyah-McDonald(1969)}
M.~F. Atiyah and I.~G. Macdonald.
\newblock {\em Introduction to commutative algebra}.
\newblock Addison-Wesley Publishing Co., Reading, Mass.-London-Don Mills, Ont.,
  1969.

\bibitem{Brown(1982)}
K.~S. Brown.
\newblock {\em Cohomology of groups}, volume~87 of {\em Graduate Texts in
  Mathematics}.
\newblock Springer-Verlag, New York, 1982.

\bibitem{Charlap-Vasquez(1969)}
L.~S. Charlap and A.~T. Vasquez.
\newblock Characteristic classes for modules over groups. {I}.
\newblock {\em Trans. Amer. Math. Soc.}, 137:533--549, 1969.

\bibitem{Cuntz-Li(2009integers)}
J.~Cuntz and X.~Li.
\newblock {$C^*$}-algebras associated with integral domains and crossed
  products by actions on adele spaces.
\newblock Preprint, arXiv:0906.4903v1 [math.OA], to appear in the Journal of
  Non-Commuttaive Geometry, 2009.

\bibitem{Cuntz-Li(2009function)}
J.~Cuntz and X.~Li.
\newblock {$K$}-theory for ring {$C^*$}-algebras attached to function fields.
\newblock Preprint, arXiv:0911.5023v1 [math.OA], 2009.

\bibitem{Davis-Lueck(2010)}
J.~Davis and W.~L{\"u}ck.
\newblock The topological ${K}$-theory of certain crystallographic groups.
\newblock Preprint, arXiv:1004.2660v1, to appear in Journal of Non-Commutative
  Geometry, 2010.

\bibitem{Lueck(2002)}
W.~L{\"u}ck.
\newblock {\em {$L\sp 2$}-{I}nvariants: {T}heory and {A}pplications to
  {G}eometry and \mbox{{$K$}-{T}heory}}, volume~44 of {\em Ergebnisse der
  Mathematik und ihrer Grenzgebiete. 3.~Folge. A Series of Modern Surveys in
  Mathematics [Results in Mathematics and Related Areas. 3rd Series. A Series
  of Modern Surveys in Mathematics]}.
\newblock Springer-Verlag, Berlin, 2002.

\bibitem{Lueck(2005s)}
W.~L{\"u}ck.
\newblock Survey on classifying spaces for families of subgroups.
\newblock In {\em Infinite groups: geometric, combinatorial and dynamical
  aspects}, volume 248 of {\em Progr. Math.}, pages 269--322. Birkh\"auser,
  Basel, 2005.

\bibitem{Lueck-Stamm(2000)}
W.~L{\"u}ck and R.~Stamm.
\newblock Computations of ${K}$- and ${L}$-theory of cocompact planar groups.
\newblock {\em $K$-Theory}, 21(3):249--292, 2000.

\bibitem{Lueck-Weiermann(2007)}
W.~L\"uck and M.~Weiermann.
\newblock On the classifying space of the family of virtually cyclic subgroups.
\newblock Preprintreihe SFB 478 --- Geometrische Strukturen in der Mathematik,
  Heft 453, M\"unster, arXiv:math.AT/0702646v2, to appear in the Proceedings in
  honour of Farrell and Jones in Pure and Applied Mathematic Quarterly, 2007.

\bibitem{Milnor(1971)}
J.~Milnor.
\newblock {\em Introduction to algebraic ${K}$-theory}.
\newblock Princeton University Press, Princeton, N.J., 1971.
\newblock Annals of Mathematics Studies, No. 72.

\bibitem{Neukirch(1999)}
J.~Neukirch.
\newblock {\em Algebraic number theory}.
\newblock Springer-Verlag, Berlin, 1999.
\newblock Translated from the 1992 German original and with a note by Norbert
  Schappacher, With a foreword by G. Harder.

\bibitem{Siegel(1995)}
S.~F. Siegel.
\newblock The spectral sequence of a split extension and the cohomology of an
  extraspecial group of order {$p^3$} and exponent {$p$}.
\newblock {\em J. Pure Appl. Algebra}, 106(2):185--198, 1996.

\end{thebibliography}

%\version{14.05.2011}

\end{document}